\documentclass[12pt,a4paper]{article}
\usepackage[margin=1.2in]{geometry}

\usepackage{stmaryrd}
\usepackage{graphicx,amssymb,amsfonts,xcolor}
\usepackage{amsmath}

\usepackage{amsthm}

\newtheorem{theorem}{Theorem}[section]
\newtheorem{lem}[theorem]{Lemma}
\newtheorem{prop}[theorem]{Proposition}
\newtheorem{corollary}[theorem]{Corollary}

\definecolor{sapgreen}{rgb}{0.31, 0.49, 0.16}

\theoremstyle{definition}
\newtheorem{defn}[theorem]{Definition}
\newtheorem{rem}[theorem]{Remark}
\newtheorem{remark}[theorem]{Remark}
\newtheorem{excont}[theorem]{Example}

\newtheoremstyle{nodot}{}{}{}{}{\bfseries}{}{ }{}
\theoremstyle{nodot}
\newtheorem*{excont*}{Example}

\newcommand{\hide}[1]{}

\newcommand{\R}{\mathbb{R}}
\newcommand{\g}{\mathfrak{g}}

\newcommand\Ad{\mathcal{A}d}
\newcommand\Am{\mathbf{Ad}}

\title{On Poisson structures arising from a Lie group action}
\author{ G.\ M.\ Beffa and E.\ L.\ Mansfield}
\date{}

\begin{document}
\maketitle

\begin{abstract}We investigate some infinite dimensional Lie algebras and their associated Poisson structures which arise from a Lie group action on a manifold. 

If $G$ is a Lie group, $\g$ its Lie algebra and $M$ is a manifold on which $G$ acts, then the set of smooth maps from $M$ to $\g$ has at least two Lie algebra
structures, both satisfying the required property to be a Lie algebroid. We may then apply a {construction} by Marle to obtain a Poisson bracket on the set of smooth 
real or complex valued functions on $M\times \g^*$. In this paper, we investigate these Poisson brackets. We show that the set of examples include 
the standard Darboux symplectic structure and the classical Lie Poisson brackets, but is a strictly larger class of Poisson brackets than these.
Our study includes the associated Hamiltonian flows and their invariants, canonical maps induced by the Lie group action, and compatible Poisson structures. 
Our approach is mainly computational and we detail numerous examples.

The Lie brackets from which our results derive, arose from the consideration of connections on bundles with zero curvature and constant torsion. 
We give an alternate derivation of the Lie bracket which will be suited to applications to Lie group actions for applications not involving a Riemannian metric. 
We also begin a study of
the infinite dimensional Poisson brackets which may be obtained by considering a central extension of the Lie algebras. 

\end{abstract}

\section{Introduction}

A recent study on post-Lie algebras by Munthe--Kaas et al \cite{MKLund}, notes in passing the construction of two distinct 
Lie algebra brackets $[\, , \,]$ and $\llbracket\, ,\,\rrbracket$ on the vector space of maps from a manifold $M$ to a Lie algebra $\g$. The second bracket depends
on a Lie group action on the manifold, $G\times M\rightarrow M$, where $\g$ is the Lie algebra of $G$. 
It is also noted that both brackets make {the space of sections of the trivial bundle $M \times \g$, denoted here by $A^1(M,M\times \g)$,} into a Lie algebroid, which means that the Lie brackets have
a Leibnitz-like property. Meanwhile, a study of the differential calculus which is possible on Lie
Algebroids, by Marle, \cite{Marle}, shows how one may construct a Poisson structure on a different but related vector space
to $A^1(M,M\times \g)$. { Linear Poisson structures associated to a Lie algebroid were first considered in \cite{CDW}.}

In the first part of this paper, we study the Poisson structures so generated by the second bracket on {$A^1(M, M\times \g)$.
 If the group action is trivial, the standard Lie Poisson structure on the dual of the Lie algebra is obtained, while if the group action is translation in the coordinates on the manifold, then the standard
Darboux symplectic structure is obtained.  We discuss canonical actions for our Poisson bracket and compare examples
of the Hamiltonian flows they generate with those of the associated Lie Poisson bracket.
 If the Lie group action is, in some neighbourhood on the 
manifold, free and regular, so that we have, locally, a moving frame, we analyse the Hamiltonian systems further, in terms of the frame and the invariants of the action on $M$.
Examples include nonlinear actions of the orthogonal group and of $SL(2)$.

We show that if the group $G$ acts on $M$ using two different actions, then the two Poisson structures are generically compatible if, and only if, the actions differ by a translation at the infinitesimal level. 

The post-Lie agebras studied in \cite{MKLund} arise in the context of Riemanninan manifolds with zero curvature but constant torsion. To expand the range of examples and perhaps extend our results to discrete spaces, we introduce a product in {$A^1(M, M\times G)$, the set of section of the trivial bundle $G\times M$},  that gives it the structure of local Lie group. We show that $A^1(M, M\times\g)$ with the second Lie bracket is the Lie algebra of this local Lie group, thus providing a geometric interpretation of $\llbracket\, , \, \rrbracket$, and a description of the symplectic leaves of its associated Lie-Poisson bracket as its coadjoint orbits.

We close the paper with a preliminary study of the infinite dimensional Lie-Poisson brackets defined by these Lie algebra brackets. In particular, we assume that $M = S^1$ or $\R$ and we show that the standard cocycle used to describe central extensions of $C^\infty(M, \g)$ is also a cocycle when we consider the second bracket, $\llbracket\, , \,\rrbracket$ instead of the pointwise bracket. This allow us to define the central extension of the new algebra. We also describe a natural compatible Poisson companion to this central extension. The study of possible applications to the integrability of PDEs will appear elsewhere.

\subsection{Basic notions}\label{SecBasic}
{While the content of this section might be elementary to members of some pure mathematical communities, this paper was inspired by \cite{LieGpNum} and our intended audience comprises the computational and application-driven mathematical communities. Thus, we include some basic definitions, and always with a computational angle.}

A \textit{Lie group} $G$ is a group which is also a manifold, such that the operations of taking products and inverses are smooth maps.
We consider Lie group actions on manifolds. 
\begin{defn}\label{ActionDef} If $M$ is a manifold and $G$ a Lie group, then we say that
the map \[ G\times M \rightarrow M,\qquad (g,z)\mapsto  g\cdot z \mbox{ (left)},\qquad g\cdot_R z \mbox{ (right)} \]
is a \textit{left Lie group action} if
\[ h \cdot (g\cdot z))= (hg)\cdot z \]
or a \textit{right Lie group action} if
\[ h \cdot_R (g\cdot_R z))= (gh)\cdot_R z \]
\end{defn}

{Given that our paper is computationally oriented,} in all our calculations we will assume local coordinates $z=(z_1, \dots, z_p)$ on $M$, and so we may take $M$ to be an open neighbourhood of $\mathbb{R}^p$ for some integer $p\ge1$.

If $G\times M \rightarrow M$ is a left action, then the induced map on the set of functions $f:M\rightarrow \mathbb{R}$, given by
\[ (g\cdot_R f ) (z) = f(g\cdot z),\]
is a right action, since $h\cdot_R(g\cdot_R f ) (z) = h\cdot_R (f(g\cdot z)) = f(g\cdot(h\cdot z)) = f((gh)\cdot z)$, and vice versa.

The \textit{Lie algebra} $\mathfrak{g}$ associated to a Lie group $G$ is the vector space $T_eG$, that is, the tangent space to $G$ at the identity element $e\in G$, which we take to be
defined in the standard way, as equivalence classes of smooth curves passing through $e$ (see Hirsch \cite{Hirsch}). 
Given $v\in T_eG= \g$, there is a distinguished representative curve, denoted as $t\mapsto \exp(t v)\in G$ and called the \textit{exponential} of the vector $v$. 
We refer to \cite{AMP,Mansbook} for the standard details.

It is well known that the space $\mathfrak{g}$ has a \textit{Lie bracket}, $[\,,\,]:\mathfrak{g}\times \mathfrak{g} \rightarrow \mathfrak{g}$ obtained via a double differentiation of the conjugation\((g,h) \to g h g^{-1}\).
In the case of a matrix Lie group, the Lie bracket on $\mathfrak{g}$ 
is the 
standard Lie matrix bracket given by $[X,Y]=XY-YX$.
Given co-ordinates $g=g(a_1, \dots , a_r)$ on $G$ near the identity, with $e=g(0,\dots 0)$ and matching basis of $T_eG$ given (as equivalence classes of curves) by
\begin{equation}\label{LieAlggensfromGcoords}v_1=[g(t,0,\dots,0)], v_2=[g(0,t,0,\dots , 0)],\dots , v_r = [g(0,\dots,0,t)]\end{equation} we may calculate
 the bracket table for the Lie algebra with respect to this basis. This yields a skew-symmetric \textit{bracket table}, 
 $B=B(\mathbf{v})$, where
 $\mathbf{v}=(v_1,v_2,\dots v_r)$, given by
$B_{i,j}=[v_i,v_j]$. { Notice that the entries of $B$ are Lie algebra elements.}

A Lie group action on a manifold $M$ yields a representation of the Lie algebra  in $\mathfrak{X}(M)$, the vector space of vector 
fields on $M$.
Let $T_zM$ be the tangent space of $M$ at the point $z\in M$. Then for $v\in T_eG=\mathfrak{g}$ define
 the vector $X_v(z)\in T_zM$ by  
\begin{equation}\label{lambdadef}
X(v)(z) = \frac{{\rm d}}{{\rm d} t}\Big\vert_{t=0} \exp(t v)\cdot z \in T_zM.
\end{equation}
The vector field $X(v)$, given by \begin{equation}\label{InfVFdef}X(v): M\rightarrow TM,\qquad  z\mapsto  X(v)(z)\end{equation}    
is called the \textit{infinitesimal vector field} associated to the vector $v\in \mathfrak{g}$. 

Given co-ordinates  $z=(z^1, z^2, \dots, z^p)$ on $M$,  and a basis $\{v_i \, |\, i=1,\dots r\}$ of $\g$,
we will write $X({v_k})$ as $X_k$ for simplicity, with 
 \begin{equation}\label{InfGensFromLieAlgBasis}X_k = \sum_{\ell} \left(\frac{{\rm d}}{{\rm d}t}\Big\vert_{t=0} {(\exp(t v_k)\cdot z)^{\ell}} \right)\partial_{z^{\ell}}
 \end{equation}
{where $(\exp(t v_k)\cdot z)^{\ell}$ denotes {the  $\ell$-th coordinate} of $\exp(t v_k)\cdot z$}.   Thus,  
 there is a $p\times r$
\textit{matrix of infinitesimals}, $\Phi$, for the Lie group action, so that the infinitesimal vector fields can be written as
\begin{equation}\label{PhiInfVFdef}
\left(\begin{array}{c} X_1\\ \vdots \\ X_r\end{array}\right)= 
\Phi^T \left(\begin{array}{c} \partial_{z^1}\\ \vdots \\ \partial_{z^p}\end{array}\right).
\end{equation}

\begin{defn}\label{AssociatedBasisInfs} We say that the infinitesimal vector fields $X_k$, $k=1,\dots, r$ defined in Equation (\ref{InfGensFromLieAlgBasis}) are \textit{associated} to the basis $v_k$, $k=1,\dots r$ of $\g$.\end{defn}
Clearly, if $v=\sum c_k v_k$, then 
\begin{equation}\label{AssIsLinear} X(v)=\frac{{\rm d}}{{\rm d}t}\Big\vert_{t=0} \exp(t \sum c_k v_k)\cdot z = \sum c_k \frac{{\rm d}}{{\rm d}t}\Big\vert_{t=0} \exp(t  v_k)\cdot z =\sum c_k X_k.\end{equation}

{In our examples, we will always use infinitesimal vector fields associated to a basis of the Lie algebra calculated using coordinates near the identity of $G$, as  in Equation (\ref{LieAlggensfromGcoords}).}

\begin{excont}\label{mySL2projEx} Consider the Lie group $SL(2)$, 
 \[ SL(2) = \left\{ g(a,b,c,d)=\left(\begin{array}{cc} a &b\\ c&d\end{array}\right) \, \Big| \, ad-bc=1\right \}\]
 acting on $M=\mathbb{R} \cup \{\infty\}$ as
 \begin{equation}\label{mySL2projact} g(a,b,c,d)\cdot u =\left\{ \begin{array}{cl}\displaystyle\frac{au+b}{cu+d} & \quad cu+d \ne 0\\ [12pt]  
 \infty & \quad cu+d=0\end{array}\right.\end{equation}
 together with $g(a,b,c,d)\cdot \infty = a/c$. For $g(a,b,c,d)$ close to the identity, we have that 
 $d=(1+bc)/a$ and so the Lie algebra is
 \begin{equation}\label{sl2LAdef} \mathfrak{sl}(2)=\left\langle v_a = \left(\begin{array}{cc}1 & 0 \\ 0 & -1\end{array}\right),
 v_b=\left(\begin{array}{cc}0 & 1 \\ 0 & 0\end{array}\right),
 v_c=\left(\begin{array}{cc}0 & 0 \\ 1 & 0\end{array}\right)\right\rangle_{\mathbb{R}}.\end{equation}
For the action in Equation (\ref{mySL2projact}), we have
\[X({v_a})=\left(\frac{{\rm d}}{{\rm d} t}\Big\vert_{t=0} \exp(t v_a)\cdot u\right)\, \partial_u = 
\left(\frac{{\rm d}}{{\rm d} t}\Big\vert_{t=0} {\frac{e^{t}u}{ e^{-t}}}\right)\, \partial_u=2u\partial_u\]
and similarly \[X({v_b})=\partial_u,\qquad X({v_c}) = -u^2 \partial_u.\]
Hence the matrix of infinitesimals is
\[\Phi = \bordermatrix{ & a&b&c\cr u & 2u & 1 & -u^2}\]
It is instructive to compare the bracket tables for the two representations of $\mathfrak{sl}(2)$. We have
\[
 \begin{array}{c|ccc} [\, ,\, ] & v_a & v_b & v_c \\ \hline
 v_a & 0& 2 v_b & -2 v_c\\ v_b & -2 v_b & 0 & v_a\\ v_c & 2 v_c& -v_a & 0
 \end{array}
\qquad 
\begin{array}{c|ccc} [\, ,\, ] & X(v_a) & X(v_b) & X(v_c) \\ \hline
  X(v_a)& 0& -2 X(v_b) & 2X(v_c)\\X(v_b) & 2 X(v_b) & 0 & -X(v_a)\\ X(v_c) & -2 X(v_c)& X(v_a) & 0
 \end{array}
\]
where the bracket on the left is the standard matrix bracket, while the bracket on the right is the standard conmutator of
vector fields, $[X,Y](f)=X(Y(f))-Y(X(f))$.
Hence the isomorphism from the matrix representation of $\mathfrak{sl}(2)$ to that in terms of the infinitesimal vector fields is
\[v\mapsto {-} X(v).\]
\end{excont}

The Adjoint action of a Lie group on its Lie algebra plays a key role in all our calculations. For a matrix Lie group, the left Adjoint action
is defined as conjugation
\begin{equation}\label{AdjDefn1}
G\times \mathfrak{g}\rightarrow \mathfrak{g},\qquad (g, x)\mapsto \Ad(g) x = g x g^{-1}.
\end{equation}
 $\Ad(g)$ is an invertible linear map from $\mathfrak{g}$ to itself and defines a representation of 
$G$ in  $\mathfrak{gl}(\mathfrak{g})$.

{The  Adjoint action of $G$  on vector fields 
can be defined by viewing the action by $g$ as a diffeomorphism of $M$. 
This group of diffeomorphisms is a Lie group and its algebra can be identified with vector fields on $M$. 
The left Adjoint action on vector fields is then defined as 
\[\Ad(g): T_zM\rightarrow T_{g\cdot z}M,\qquad {\Ad(g) (X) (z) = Tg \circ X (g^{-1}z)}\]
so that the following diagram commutes,
\[ \begin{array}{ccl} TM & \stackrel{Tg}{\rightarrow} & TM \\
    X \uparrow & & \uparrow \Ad(g)(X)\\
    M & \stackrel{g\cdot }{\rightarrow} & M 
   \end{array}
\]
}
In coordinates, if we write $\nabla_z = (\partial_{z^1}, \partial_{z^2},\dots, \partial_{z^p})^T$ then we have that for a vector field $X=\sum f_j(z)\partial_{z^k}= \mathbf{f}(z)^T\nabla_z$
\[ \Ad(g) (X) = Tg\left(\mathbf{f}^T(g^{-1}\cdot z)\nabla_z|_{g^{-1}\cdot z}\right)= \mathbf{f}^T(g^{-1}\cdot z)\left(\frac{\partial \left(g^{-1}\cdot z\right)}{\partial z}\right)^{-T}\nabla_z
\]
\[
=\left(\left(\frac{\partial \left(g^{-1}\cdot z\right)}{\partial z}\right)^{-1}\mathbf{f}(g^{-1}\cdot z)\right)^T\nabla_z
\]
The representation of the Adjoint linear action on the coefficients $\mathbf{f}$ is thus given by 
\begin{equation}\label{CoeffsTransAdj}
\mathbf{f}(z) \mapsto \left(\frac{\partial \left(g^{-1}\cdot z\right)}{\partial z}\right)^{-1}\mathbf{f}(g^{-1}\cdot z).
\end{equation}

We have that if $X(v)$ is the field generated by $v\in \g$, then $\Ad(g)(X(v)) = X(\Ad(g)(v))$.
Indeed, we have  for $X=X(v_k)$ that
\[ \begin{array}{rcl}
\Ad(g)(X(v_k))&=&Tg \circ X \circ g^{-1} (z)\\ &=& Tg \frac{{\rm d}}{{\rm d}t}\Big\vert_{t=0} \exp(t v_k) \cdot g^{-1}\cdot z\\
&=& \frac{{\rm d}}{{\rm d}t}\Big\vert_{t=0} g \exp(t v_k) g^{-1} \cdot z\\
&=& \frac{{\rm d}}{{\rm d}t}\Big\vert_{t=0}  \exp(t (gv_kg^{-1}) ) \cdot z\\
&=& X(\Ad(g)(v_k)).\end{array}
\]
The result then follows by linearity. 

Given a basis $v_k$, $k=1,\dots, r$ of $\g$, we know that the Adjoint of a vector in $\g$ is a linear combination of these basis vectors. Using Equation (\ref{InfGensFromLieAlgBasis}), if $\Ad(g)(\sum f_k v_k)=\sum v_k \Am(g) f_j$, 
where $\Am(g)$ is the matrix representation of $\Ad(g)$, we have for the associated infinitesimal vector fields  that
$ \Ad(g) \sum f_kX_k = \sum X_k \Am(g) f_j$  or \begin{equation}\label{PhiTransAdj}\Phi^T\nabla_z \mapsto \Am^T(g)\Phi^T\nabla_z .\end{equation}
Combining this with Equation (\ref{CoeffsTransAdj}) yields
\begin{equation}\label{PhiUltimateEqn}
\left(\frac{\partial \left(g^{-1}\cdot z\right)}{\partial z}\right)^{-1}\Phi(g^{-1}\cdot z)= \Phi(z)\Am(g).
\end{equation}

We now illustrate that using associated infinitesimal vector fields leads to the same matrix representation of the Adjoint matrix $\Am(g)$,

\begin{excont*} \textbf{\ref{mySL2projEx} (cont.)} For the action in equation (\ref{mySL2projact}), recall
\[\Phi = \bordermatrix{ & a&b&c\cr u & 2u & 1 & -u^2}.\]
We have
\[ g(a,b,c,d)^{-1}\cdot u = \displaystyle\frac{du-b}{-cu+a},\qquad 
\left(\frac{\partial (g^{-1}\cdot u)}{\partial u}\right)^{-1} \frac{\partial}{\partial u}
=(-cu+a)^2 \frac{\partial}{\partial u}
\]
so that, for example
\[\Ad(g)(X_a) = 2 (du-b)(-cu+a)\partial_u= (ad+bc)X_a -2ab X_b +2cd X_c.\]

The complete calculation gives
\begin{equation}\label{AdgTSL2}\Phi^T\partial_u= \left(\begin{array}{c} X_a\\ X_b \\ X_c \end{array}\right)\mapsto 
\left(\begin{array}{ccc} ad+bc& -2ab & 2cd\\ -ac & a^2 & -c^2 \\ bd& -b^2 & d^2\end{array}\right)
\left(\begin{array}{c} X_a\\ X_b \\ X_c \end{array}\right)=\Am(g)^T\Phi^T\partial_u\end{equation}
where, as above, this defines the matrix $\Am(g)$.

Let's set
\[ g=\left(\begin{array}{cc} a&b\\ c&d\end{array}\right)\in G,\qquad x(\alpha,\beta,\delta)
=\left(\begin{array}{cc} \alpha & \beta \\ \delta & -\alpha
\end{array}\right)\in \g \]
into $ \Ad(g)(x)= g x g^{-1}=x(\widetilde{\alpha},\widetilde{\beta},\widetilde{\delta})$, the expression defining
$\widetilde{\alpha}$, $\widetilde{\beta}$, $\widetilde{\delta}$. We can compare the Adjoint action above with that on the matrix representation of $\mathfrak{sl}(2)$, given in
equation (\ref{sl2LAdef}), observing that
\[ \left(\begin{array}{c} \widetilde{\alpha}\\ \widetilde{\beta}\\ \widetilde{\delta}\end{array}\right)
=\left(\begin{array}{ccc} ad+bc & -ac & db \\ -2ab & a^2 & -b^2\\ 2cd & -c^2 & d^2 \end{array}\right)
\left(\begin{array}{c} {\alpha}\\ {\beta}\\ {\delta}\end{array}\right)
=\Am(g)\left(\begin{array}{c} {\alpha}\\ {\beta}\\ {\delta}\end{array}\right)\]
which is precisely the same matrix representation previously obtained.
\end{excont*}

\begin{defn} Let $\mathcal{C}^{\infty}(M,\mathbb{R})$ denote the set of smooth real-valued functions on $M$. A \textit{Poisson bracket} is a map,
\[ \{\, , \, \} : \mathcal{C}^{\infty}(M,\mathbb{R})\times \mathcal{C}^{\infty}(M,\mathbb{R}) \rightarrow \mathcal{C}^{\infty}(M,\mathbb{R}),\qquad
(F, G)\mapsto \{ F,G\} \]
is called a Poisson bracket if for all $a$, $b\in \mathbb{R}$, $F$, $G$ $H\in \mathcal{C}^{\infty}(M,\mathbb{R})$
\begin{enumerate}
\item $ \{ F, G\}=-\{G,F\}$ (skew-symmetry)
\item $\{ aF +b G, H\}=a\{F,H\}+b \{G,H\}$, (bilinearity)
\item $0=\{F,\{G,H\}\}+\{G,\{H,F\}\}+\{H,\{F,G\}\}$ (Jacobi identity)
\item $\{F,GH\}= \{F,G\}H + G\{F,H\}$ (Leibnitz identity)
\end{enumerate}
Given co-ordinates $z=(z^1,\dots z^p)$ on $M$, then viewing these co-ordinates $z^k$ as functions on $M$, we may define the 
\textit{Poisson structure matrix} $\Lambda$ for the Poisson bracket as 
\[\Lambda=(\Lambda_{ij}),\qquad  \Lambda_{ij} = \{ z^i, z^j\}.\]\end{defn}

{If $M=\g^\ast$, then $M$ has a distinguished Poisson bracket called the Lie Poisson bracket. Specifically, let
$\{v_1, \dots v_r\}$ be a basis of $\mathfrak{g}$, and $\{\theta^1, \dots, \theta^r\}$ its dual basis in $\mathbf{g}^*$ so that $\theta^i(v_j) = \delta_i^j$.  If $f\in C^\infty(\g^\ast)$, we define the variational derivative of $f$ at $\xi\in\g^\ast$, and we denote it by $\delta_\xi f$ (or $\delta f(\xi)$) as the {\it element of $\g$} such that 
\[
\frac d{d\epsilon}|_{\epsilon=0} f(\xi+\epsilon \nu) = \nu (\delta_\xi f)
\]
for all $\nu\in \g^\ast$.} 

{Define the \textit{Lie Poisson} bracket on  $C^\infty(\mathfrak{g}^*)$ by the formula
 \[ 
 \{ f, g \}(\xi)= \xi \left([ \delta_\xi f, \delta_\xi g]\right).
 \]
} 
 {Assume $\delta_\xi f =  \sum_i v_i\partial_{i} f $, where $\partial_{i} f$ are defined by this relationship. Then, if $[v_i, v_j] = \sum_k c_{ij}^k v_k$,
 the Lie-Poisson bracket can be written as
 \begin{equation}\label{L-P}
  \{ f, g \}(\xi)= \xi \left(\sum_{i,j,k} c_{i,j}^k v_k\partial_{i} f \partial_{j} g \right),
 \end{equation}
 or
 \[ \{ f, g \}(\theta^k)= \sum_{i,j}c_{i,j}^k \partial_{i} f \partial_{j} g. 
 \] 
 Denote by $\xi_i$,$i=1,\dots r$ the linear coordinate functions on $\g^\ast$ representing $v_i$, $i=1,\dots r$, thinking of $v_i$ as the dual to $\theta^i$, so that $\xi_i(\xi) = v_i(\xi)$; we say in this case that the basis $\{\xi_i\}$ is
 \textit{associated} to the basis $\{ v_i\}$. Then
 \begin{equation}\label{dualdualbr}
 \{\xi_i, \xi_j\}(\theta^k) = c_{i,j}^k ~~ \text{or}~~ \{\xi_i, \xi_j\} = \sum_k c_{i,j}^k \xi_k ~~\text{which is the bracket in $\g$}.
 \end{equation}
 }

We now note the following two interesting facts which we will use in the sequel. {The first is a well-known
change of variable in our computationally oriented notation.} 
Let $\Am(g)$ be the matrix representation of the adjoint action as in (\ref{PhiUltimateEqn}).

{If the
Adjoint action on the Lie Algebra is given in matrix form as $(v_1, \dots, v_r)\mapsto (v_1, \dots, v_r)\Am(g)$ then
the induced action on $\mathfrak{g}^{**}$ 
is given by 
\begin{equation}\label{AdjXiAct}\boldsymbol{\xi}\mapsto \boldsymbol{\xi} \Am(g).  \end{equation}
where $\boldsymbol{\xi} = (\xi_1,\dots,\xi_r)$ and the $\xi_i$ are coordinates on $\mathfrak{g}^{*}$.} 
\begin{prop}[The induced Adjoint action on the Lie Poisson structure]
\label{prop3things}
Let $G$ a Lie group with $\dim G=r$, with $\mathfrak{g}$ its Lie algebra, and with $\g^{**}$ having basis $\xi_1, \dots, \xi_r$, viewed as coordinates on $\g^\ast$ associated to the basis of $\g$,  $v_1, \dots, v_r$ . Using these coordinates, let the  Lie Poisson structure matrix for $\mathfrak{g}^*$ be denoted as $\Lambda=\Lambda(\boldsymbol{\xi})$.
\begin{itemize}
\item[1.]
The Lie Poisson structure matrix transforms as
\[ \Lambda({\boldsymbol{\xi}}\Am(g) ) = \Am(g)^T \Lambda(\boldsymbol{\xi}) \Am(g).\]
\item[2.] The infinitesimal matrix  for the {co-}Adjoint action on $\mathfrak{g}^*$ is the negative of the Lie Poisson structure matrix.
\end{itemize}
\end{prop}

\begin{proof}
\begin{itemize}
\item[1.]  The first result is a direct linear algebra consequence of (\ref{dualdualbr}) and (\ref{AdjXiAct}).  
\item[2.] 
{Let $g(t)=\exp(tv_i)$. Then from the definition of the Adjoint action, we have
\[ \begin{array}{rcl}
\frac{{\rm d}}{{\rm d}t} \Big\vert_{t=0} \Ad(\exp(tv_i) ) v_j &=& \frac{{\rm d}}{{\rm d}t} \Big\vert_{t=0}
\exp(tv_i) v_j \exp(-tv_i)\\
&=& [v_i,v_j]\\
&=& \sum_{\nu} c^{\nu}_{ij}v_{\nu}\end{array}. \] 
By (\ref{AdjXiAct}), the same transformation rules apply, namely
\[ \frac{{\rm d}}{{\rm d}t} \Big\vert_{t=0} \Ad(\exp(tv_i) ) \xi_j = \sum c^{\nu}_{ij}\xi_{\nu} = \{ \xi_i, \xi_j\},\]
as described in (\ref{dualdualbr}). By definition of the infinitesimal matrix,
\[ (\Phi^T)_{ij}= \frac{{\rm d}}{{\rm d}t} \Big\vert_{t=0} \Ad(\exp(tv_i) ) \xi_j = \{ \xi_i, \xi_j\} \]
 giving $\Phi_{ij}=\{ \xi_j, \xi_i\}=-\{\xi_i,\xi_j\}$
as required. 
}
\end{itemize}
\end{proof}

\section{{Poisson brackets on Lie algebroids}}

{If
$E$ is a bundle with base space $M$ and fibre $\mathfrak{g}$, 
we denote the space of sections of $E$ as $A^1(M,E)$, following  Marle \cite{Marle}. Notice that $A^1(M, M\times \g)$ is 
isomorphic to $C^\infty(M, \g)$. 
As all our calculations are local, so that the 
subtleties of global obstructions do not arise in this paper, we will restrict to the trivial case.}

\subsection{Two Lie algebra structures in $C^\infty(M, \g)$}
We recall two definitions of Lie brackets on $C^1(M,\mathfrak{g})$, given in \cite{MKLund}.

\begin{defn}[The first Lie bracket]
Let $x,y:M\rightarrow \mathfrak{g} \in C^1(M,\mathfrak{g})$ and let $[\, , \, ]$ denote the Lie bracket on $\mathfrak{g}$. Then we  define
the first, pointwise Lie bracket on $C^1(M,\mathfrak{g})$, to be
\begin{equation}\label{pointwisebracketDef}\phantom{} [ x , y ](z) = [x (z), y(z)].
\end{equation}
\end{defn}

\begin{remark} Notice that if we use the representation of the Lie algebras as infinitesimal vector fields, that is, as first order operators, then
the bilinear property for this first bracket
\[ \left[\textstyle \sum x^i(z) v_i, \sum y^j(z) v_j\right] = \textstyle\sum x^i(z)y^j(z)[v_i,v_j] \]
still applies, even though
this representation is not linear with respect to multiplication of functions.
For this reason when calculating this first bracket,
we will  be using a (faithful) matrix representation. It should be noted that the isomorphism between the two representations is
\[ v\mapsto -X(v).\]
\end{remark}
A second Lie bracket may be defined in terms of a given Lie group action, $G\times~ M\rightarrow~M$. Let $v_1, \dots, v_r$ be a basis of $\mathfrak{g}$, 
and $X(v_i) = X_i$, $i=1, \dots,r$, their infinitesimal vector fields. We write these as
\[ (X_1, \dots, X_r)^T=\Phi(z)^T \nabla_z\]
where $\Phi$ is the matrix of infinitesimals given in Equation (\ref{PhiInfVFdef}).

Now let  $x: M\rightarrow \mathfrak{g}$, be given by $x(z)=\sum x^j(z) v_j$. We  define the vector field $\rho(x)$ on $M$ to be 
\begin{equation}\label{rhoxdef}
\rho(x)(z) = {X(x(z))} = \sum x^j(z) X_{j} = (\Phi\mathbf{x})^T(z) \nabla_z.
\end{equation}
Next, given a second map $y:M\rightarrow \mathfrak{g}$, $y=\sum y^j v_j$, define
\begin{equation}\label{xTriydef}
 \mathcal{L}_{\rho(x)} y = \sum \rho(x)(y^j) v_j
\end{equation}
the component-wise Lie derivative of $y$ along the vector field $\rho(x)$.  We note that $\mathcal{L}_{\rho(x)}$ is a derivation
on $C^1(M,\g)$.


We are now in a position to define a second Lie bracket. 

\begin{defn}[The second Lie bracket] If the action of $G$ on $M$ is a left action, then we define the \textit{second} Lie bracket on $A^1(M,\mathfrak{g})$, to be 
\begin{equation}\label{leftddblbracketDef} \llbracket x , y \rrbracket =  \mathcal{L}_{\rho(x)} y -  \mathcal{L}_{\rho(y)} x - [x , y]
\end{equation}
where the final summand is the first, point wise Lie bracket defined above.
If the action of $G$ on $M$ is a right action, we define the bracket to be
\begin{equation}\label{rightddblbracketDef}  \llbracket x , y \rrbracket = -\mathcal{L}_{\rho(x)} y +  \mathcal{L}_{\rho(y)} x - [x , y].
\end{equation}
\end{defn}

The Jacobi identity may be verified using the interpretation of this bracket in \S\ref{SecAlggeomBr} as the Lie bracket associated to a certain local Lie group. 

\begin{rem} If the action of $G$ on $M$ is a right action, $(g, z)\mapsto g\cdot_R z$, we can convert it to be a left action by 
considering $g\cdot z = g^{-1}\cdot_R  z$. This reverses the sign of 
$\Phi$ and in this sense, the two definitions, (\ref{leftddblbracketDef}) and (\ref{rightddblbracketDef}), are consistent.
\end{rem}
\begin{rem} As we have defined this bracket, in coordinates, it is important to note that the same coordinates on $G$ near the identity $e\in G$ are used to
 obtain both the basis vectors of the (faithful) representation $\g$ and the infinitesimal vector fields.
\end{rem}
\begin{rem}
 It can be seen that if the group action is trivial, that is, $g\cdot z = z$ for all $g\in G$, then we recover the
negative of the pointwise bracket, (\ref{pointwisebracketDef}) which is consistent with our choice of the matrix representation
to calculate that bracket.
\end{rem}

\subsection{A Lie algebroid structure on $C^\infty(M, \g)$}
{I've removed all mention of $A^1$ from here now.}
{The Lie brackets of the previous subsection make $A^1(M,M\times \mathfrak{g}) \simeq C^\infty(M, \g)$ into a \textit{Lie algebroid}}. 

\begin{defn}[Lie algebroid]\label{anchorDef}
{
We say that $C^1(M,\g)$ is a Lie algebroid if there is a bundle map $\rho$}, that is, a map preserving the base point, called the {\em anchor\/} map,
$$\rho:C^1(M,\g) \rightarrow \mathfrak{X}(M)$$
where $\mathfrak{X}(M)$ is the set of smooth vector fields on $M$,  such that for $x$, $y\in C^1(M,\g)$, and  $f:M\rightarrow \mathbb{R}$ a smooth map,
\begin{equation}\label{defnAnchor} \llbracket x, f y \rrbracket = f \llbracket x,  y \rrbracket  + \mathcal{L}_{\rho(x)}(f )\ y.\end{equation}
\end{defn}

\begin{rem}\label{RemRhoLieAlgHomSB}The definition implies that
$$ \rho\left( \llbracket x,  y \rrbracket \right)= \left[ \rho(x), \rho(y)\right]$$
where the bracket on the right is the bracket of vector fields in $\mathfrak{X}(M)$. In other words, the anchor map is a Lie algebra homomorphism from $C^1(M,\g)$ to $\mathfrak{X}(M)$, which preserves the base point, (cf.\ \cite{Kosman}).

Indeed, consider \[ \begin{array}{rcl}
\llbracket z, \llbracket x, fy\rrbracket\rrbracket &=& \llbracket z , f \llbracket x,  y \rrbracket  + \mathcal{L}_{\rho(x)}(f )\ y\rrbracket \\
&=&\llbracket z , f \llbracket x,  y \rrbracket  +\llbracket z , \mathcal{L}_{\rho(x)}(f )\ y\rrbracket \\
&=& f \llbracket z, \llbracket x,y\rrbracket\rrbracket + \mathcal{L}_{\rho(z)}(f )\ \llbracket x, y\rrbracket+\mathcal{L}_{\rho(x)}(f )\ \llbracket z, y\rrbracket +  \mathcal{L}_{\rho(z)}( \mathcal{L}_{\rho(x)}(f ))y
\end{array}
\]
applying Equation (\ref{defnAnchor}) to each summand. Similarly, we have,
\[ \llbracket x, \llbracket z, fy\rrbracket\rrbracket  = f \llbracket x, \llbracket z,y\rrbracket\rrbracket + \mathcal{L}_{\rho(x)}(f )\ \llbracket z, y\rrbracket\\+\mathcal{L}_{\rho(z)}(f )\ \llbracket x, y\rrbracket +  \mathcal{L}_{\rho(x)}( \mathcal{L}_{\rho(z)}(f))y\]
and further, we have,
\[ \llbracket \llbracket x,z\rrbracket , fy\rrbracket = f \llbracket \llbracket x,z\rrbracket , y\rrbracket +  \mathcal{L}_{\rho(\llbracket x,y\rrbracket)}(f )\ y.\]
Applying the Jacobi identity in the form,
\[\llbracket x , \llbracket z, fy \rrbracket\rrbracket -\llbracket z , \llbracket x , fy \rrbracket\rrbracket =\llbracket \llbracket x,z \rrbracket ,fy \rrbracket  \]
yields the result,
\[   \mathcal{L}_{\rho(x)}( \mathcal{L}_{\rho(z)}(f ))- \mathcal{L}_{\rho(z)}( \mathcal{L}_{\rho(x)}(f ))= \mathcal{L}_{\rho(\llbracket x,z\rrbracket)}(f ) \]
as desired.
\end{rem}

The anchor map allows a bundle $E$ to be viewed as a proxy for $T(M)$.
A great deal more can be said, indeed the anchor map allows for proxy Lie derivatives, exterior derivatives and many other constructions, \cite{Marle}. 

The pointwise Lie bracket (\ref{pointwisebracketDef}) makes $C^\infty(M,\g)$ a Lie algebroid with a zero anchor map. The second Lie bracket, (\ref{leftddblbracketDef}) or (\ref{rightddblbracketDef}), 
makes $C^\infty(M,\g)$ a Lie algebroid with the anchor map being precisely $\rho(x)$,  Equation (\ref{rhoxdef}), and hence the reason for our choice of notation.

In the examples which follow, we use the standard coordinate names arising in the applications.

\begin{excont}\label{RunEgTrans}
We set $M=\mathbb{R}$ with coordinate labelled as $t$, and $G=(\mathbb{R}, +)$, that, the real numbers under addition, with the group action being
$$ \epsilon\cdot t = t+\epsilon.$$ Then the Lie algebra is $\mathfrak{g}=\mathbb{R}$ with the trivial Lie bracket and the single
 infinitesimal vector field is $\partial_t$. Hence for $x:M\rightarrow \g$, $\rho(x)=x(t)\partial_t$ and thus
 for $x, y:M \rightarrow \mathfrak{g}$, we have
$$ \llbracket x , y \rrbracket  = x y_t - y x_t.$$
The proof of the Jacobi identity is straightforward given the maps into $\mathfrak{g}$ are commuting scalars.
\end{excont}

\begin{excont*} \textbf{\ref{mySL2projEx} (cont.)}
Recall the projective action of $SL(2)$ on $M=\mathbb{R}$ given in Equation (\ref{mySL2projact}) and the Lie algebra $\mathfrak{sl}(2)$ 
given in Equation (\ref{sl2LAdef}).

A map $x:M\rightarrow \mathfrak{g}$ takes the form
$$ u\mapsto x(u) =\left(\begin{array}{cc} x^1(u) & x^2(u)\\ x^3(u) & -x^1(u)\end{array}\right) .$$
Given two such maps $x$, $y$ we have
$$ \mathcal{L}_{\rho(x)}y = \left(\begin{array}{cc}  {\rho(x)} y^1(u) &  {\rho(x)} y^2(u)\\  {\rho(x)}  y^3(u) & - {\rho(x)} y^1(u)\end{array}\right) $$
where
$$\rho(x) 
= \left(2u x^1(u) +x^2(u) -u^2 x^3(u) \right)\partial_u.$$

It can be verified directly that the second bracket (\ref{leftddblbracketDef}), given explicitly in this case using the
standard matrix representation as,
\begin{equation}\label{ProjActDblBr}
\begin{array}{rcl}  \llbracket x, y\rrbracket &=&  \left(2u x^1(u) +x^2(u) -u^2 x^3(u) \right)
 \left(\begin{array}{cc} y^1_u & y^{2}_{u}\\ y^{3}_{u} & - y^1_u\end{array}\right) \\[15pt]&& \quad -
\left(2u y^1(u) +y^2(u) -u^2 y^3(u) \right) \left(\begin{array}{cc} x^{1}_{u} & x^{2}_{u}\\ x^{}_{u} & - x^{1}_{u}\end{array}\right)\\[15pt]
&&
\quad - x(u)y(u)+y(u)x(u).
\end{array}
\end{equation}
satisfies the Jacobi identity.
\end{excont*}

\begin{excont}\label{RunEgLeftSE2} We consider $G=SE(2)=SO(2)\ltimes \mathbb{R}^2$.
A matrix representation of $G=SE(2)$ is
\[ g(\theta,a,b)= \left(\begin{array}{ccc} \cos\theta & -\sin\theta & a\\ \sin\theta & \cos\theta & b\\ 0&0&1\end{array}\right)\]
so that the Lie algebra is
\[ \mathfrak{se}(2)=\left\{ \alpha v_{\theta}+\beta v_a + \delta v_b=
\left(\begin{array}{ccc}0 &-\alpha & \beta\\ \alpha & 0 & \delta\\0&0&0
 \end{array}\right)\, | \, \alpha, \beta, \delta \in \mathbb{R}\right\}\]
 where this defines $v_{\theta}$, $v_a$ and $v_b$.
We consider the standard, left linear action of 
$SE(2)$ on $M=\mathbb{R}^2$ with coordinates $x$ and $y$, given as
\begin{equation}\label{SE2Act}g(\theta,a,b)\cdot \left(\begin{array}{c} x\\y\end{array}\right) 
= \left(\begin{array}{cc} \cos \theta & -\sin \theta \\ \sin \theta & \cos \theta\end{array}\right) 
\left(\begin{array}{c} x\\y\end{array}\right)  +  \left(\begin{array}{c} a\\b\end{array}\right).\end{equation}
The infinitesimal vector fields are
\[ X(v_{\theta})= -y\partial_x+x\partial_y,\qquad X(v_a)=\partial_x,\qquad X(v_b)=\partial_y.\]
A map $\chi:M\rightarrow \mathfrak{se}(2)$, is then
$\chi(x,y)=\chi^1(x,y) v_{\theta}+\chi^2(x,y)v_a + \chi^3(x,y)v_b$.

We have that $$\rho(\chi) = \chi^1(x,y) \left(x \partial_y - y \partial_x\right) + \chi^2(x,y) \partial_x + \chi^3(x,y)\partial_y$$
so that 
$$\begin{array}{rcl} \llbracket \chi,\eta\rrbracket &=& \left(\chi^1\left(x \partial_y - y \partial_x\right) + \chi^2 \partial_x + \chi^3\partial_y\right)\eta \\
&&\quad - \left(\eta^1\left(x \partial_y - y \partial_x\right) + \eta^2 \partial_x + \eta^3\partial_y\right)\chi\\
&&\quad - \chi\eta+\eta\chi
\end{array}$$
where the operators act component-wise. It is straightforward to verify this bracket satisfies the Jacobi identity.

\end{excont}

\begin{excont}\label{RunEgleftSO3}
 We consider the standard action of $SO(3)$, the special orthogonal group in dimension 3, acting linearly on $\mathbb{R}^3$. 
 With respect to the standard coordinates, $(x,y,z)$ on $\mathbb{R}^3$, the infinitesimal 
 vector fields  can be taken to be
 \[ \begin{array}{rcl} X(v_{xy}) &=& x \partial_y - y \partial_x\\X(v_{yz}) &=& y\partial_z - z\partial_y\\
     X(v_{zx}) &=&  z\partial_x- x\partial_z 
    \end{array}\]
This implicitly defines three coordinates on $SO(3)$ near the identity, and with respect to these coordinates, the matching basis of
the Lie algebra $\mathfrak{so}(3)$ is 
\[ v_{xy} = \left(\begin{array}{ccc} 0&-1&0\\1&0&0\\0&0&0\end{array}\right),\qquad 
v_{yz} = \left(\begin{array}{ccc} 0&0&0\\0&0&-1\\0&1&0\end{array}\right),\qquad
 v_{zx} = \left(\begin{array}{ccc} 0&0&1\\0&0&0\\-1&0&0\end{array}\right)
 .\]
 A map $\chi:\mathbb{R}^3\rightarrow \mathfrak{so}(3)$ takes the form $\chi(x,y,z)=\chi^{xy}(x,y,z) v_{xy}+\chi^{yz}(x,y,z) v_{yz}
 +\chi^{zx}(x,y,z)v_{zx}$, and
 \[ \rho(\chi) = \chi^{xy} X(v_{xy})+\chi^{yz} X(v_{yz})+\chi^{zx} X(v_{zx}).\]
 Given two maps $\chi, \eta : \mathbb{R}^3\rightarrow \mathfrak{so}(3)$, then
 \[  \llbracket \chi, \eta \rrbracket = \rho(\chi)\eta - \rho(\eta)\chi - \chi\eta+\eta\chi \]
 defines a Lie bracket.  It is straightforward to verify this bracket satisfies the Jacobi identity. 
\end{excont}

\subsection{Poisson structures arising from Lie algebroids}

 We now outline the construction of a Poisson structure based on a result in Marle \cite{Marle}. The relation between the dual of vector bundles and Poisson structures can also be found in \cite{CW}, \cite{DZ} and \cite{Mac}.  {Because our approach is local and computational, we continue to restrict to the case of the trivial bundle, $M\times \g$.}
 
Consider the bundle $E^*=M\times \mathfrak{g}^*$ where $\mathfrak{g}^*$ is the dual of $\mathfrak{g}$. 
For each smooth section $x\in C^\infty(M,\g)$ we associate the smooth function
\[ \varphi_x~:~ E^*~\rightarrow~ \mathbb{R},\qquad
\varphi_x(z,\theta) =\theta(x(z)).\]
Now let $z=(z^1,\dots, z^p)$ be coordinates on $M$ and let $\xi=(\xi^1,\dots, \xi^r)$ be coordinates on $\mathfrak{g}^*$. 
\begin{defn}[Section and function associated to a function on $E^\ast$]
{Let $\eta: E^\ast \to \R$.  We say the 
section $x_{\eta}\in C^\infty(M,\g)$ and the function $f_{\eta}:M\rightarrow\mathbb{R}$  are \textit{associated to $\eta$ at the point $(z,\xi)$}, if
$$ d\eta = d(\varphi_{x} + f \circ \pi),$$
where $\pi$ is the projection on $E^\ast$.}
\end{defn}
Associated sections and functions are not unique and always exist. In applying this definition to our class of examples, the associated sections and functions can be taken to be constant and linear, respectively. The following result derives from applying a more general result, appearing in the
course of the proof of \cite{Marle} Theorem 5.3.2, to our class of Lie Algebroids. 
\begin{theorem}[Poisson structure on $E^*$] \label{PSmarle}
{Let $E$ be a vector bundle, a Lie algebroid with anchor map $\rho$ and Lie bracket $[\, ,\, ]$. There exists a Poisson bracket on $E^\ast$ such that
\[
\{\varphi_{x_1}, \varphi_{x_2}\} = \varphi_{[x_1, x_2]}.
\]
Furthermore, given two functions on $E^*$, say $\eta_1$ and $\eta_2$, let $x_1$, $f_1$ and $x_2$, $f_2$ be the  sections  and functions associated to $\eta_1$ and $\eta_2$ respectively. The Poisson bi-vector associated to $\{,
\}$ is given by
$$ \Lambda(z,\xi) (\eta_1,\eta_2)= \varphi_{[x_1,x_2]} (z,\xi)+\mathcal{L}_{\rho(x_1)}(f_2)(z) - \mathcal{L}_{\rho(x_2)}(f_1)(z).$$ }
\end{theorem}

Despite the non-uniqueness of the associated sections and functions, $\Lambda$ is uniquely defined. 

The calculation of the Poisson structure associated to any given Lie algebroid $E^\ast$ is straightforward, as we will now demonstrate.

\begin{excont*}\textbf{\ref{mySL2projEx} (continued). } 

{Assume the coordinates on $\mathfrak{g}^*$ are $\xi=(\xi^1, \xi^2, \xi^3)$.
Consider the functions $\eta_1$ and $\eta_2$ on $M\times\g^\ast$ such that} $x_i: M\rightarrow \mathfrak{g}$ are the constant maps, 
\[ x_i(u)=\left(\begin{array}{cc} \beta_i^1 & \beta_i^2\\ \beta_i^3 & - \beta_i^1\end{array}\right) \]
and the associated functions to be linear, $f_1(u)=\alpha_1 u$, $f_2(u)=\alpha_2 u$. 
Since the maps $x_i$ are constant, we have $\llbracket x_1,x_2\rrbracket ={-} [ x_1, x_2 ]$, the standard matrix bracket in $\mathfrak{sl}(2)$,
and 
\[ \begin{array}{rcl} \mathcal{L}_{\rho(x_1)}(f_2(u))&=&\alpha_2 \left( 2u\beta_1^1 + \beta_1^2 - u^2 \beta_1^3\right)\\
 \mathcal{L}_{\rho(x_2)}(f_1(u))&=&\alpha_1 \left( 2u\beta^1_2 + \beta^2_2 - u^2 \beta_2^3\right)\\[5pt]
 \varphi_{[x^1,x^2]}(u,\xi)&=&-\left(\left(\beta_1^1\beta_2^2 - \beta^2_1\beta_2^1\right) 2\xi^2 +\left(\beta_1^1\beta_2^3-\beta_1^3\beta^1_2\right)(-2\xi^3)\right.\\[5pt]
 &&\quad \left.+\left(\beta_1^2\beta_2^3 - \beta_1^3\beta_2^2\right)\xi^1\right)
 \end{array}
 \]
 Finally, we  have
 \[\begin{array}{rcl}
 \Lambda(u,\xi)(\eta_1,\eta_2)&={-}&\left(\begin{array}{cccc} \alpha_1&\beta_1^1& \beta_1^2&\beta_1^3\end{array}\right)
 \left(\begin{array}{cccc} 0& 2u&1&-u^2\\-2u & 0& 2\xi^2 & -2\xi^3\\-1&-2\xi^2 &0&\xi^1\\ \phantom{-}u^2&2\xi^3&-\xi^1&0\end{array}\right)
 \left(\begin{array}{c} \alpha_2\\\beta_2^1\\\beta_2^2\\ \beta_2^3\end{array}\right)\\ [15pt]
 &=&- \left(\begin{array}{cc} \alpha_1&\beta_1\end{array}\right) [ \Lambda]  \left(\begin{array}{cc} \alpha_2&\beta_2\end{array}\right)^T
 \end{array}
 \]
where this defines the skew-symmetric matrix, $[\Lambda]$.

{From this example, we see that the Poisson bracket on $E^*$} can be defined as follows: given functions $F, G:E^*\rightarrow \mathbb{R}$, we define
 \[ 
 \left\{ F,G\right\} = \left(\begin{array}{cccc} F_u& F_{\xi_1}& F_{\xi_2}& F_{\xi_3}\end{array}\right)
 [\Lambda]
 \left(\begin{array}{cccc} G_u & G_{\xi_1}  & G_{\xi_2} &  G_{\xi_3} \end{array}\right)^T
 \]
 Indeed, skew-symmetry and the Liebnitz rule are self-evident, while it is straightforward to check the Jacobi identity
 \[ \{ F \{ G,H\}\} + \{G,\{H,F\}\}+\{H,\{F,G\}\}=0.\]
\end{excont*}
Looking more closely at the above example, we see that the matrix $[\Lambda]$ has a striking structure, indeed, we have
\begin{equation}\label{newPBprojSL2}
[\Lambda]= \left(\begin{array}{c|ccc} 0& 2u&1&-u^2\\ \hline -2u & 0& 2\xi^2 & -2\xi^3\\-1&-2\xi^2 &0&\xi^1\\ \phantom{-}u^2&2\xi^3&-\xi^1&0\end{array}\right)
=
\left(\begin{array}{c|c} 0 & \Phi\\ \hline -\Phi^T &\phantom{\Big\vert} \Lambda({\mathfrak{sl}(2)^*)}\end{array}\right)
\end{equation}
where $\Phi$ is the matrix of infinitesimals, and $\Lambda(\mathfrak{g}^*)$ is the structure matrix for standard Lie Poisson bracket on $\mathfrak{g}^*$
in its matrix representation. 

The following result is immediate from Theorem \ref{PSmarle}. 
\begin{theorem}[Poisson structure on {$M\times \mathfrak{g}^*$}]\label{FinDimPB}
Let $G\times M\rightarrow M$ be a left action of the Lie group $G$ on the manifold $M$, with matrix of infinitesimals, $\Phi$. 
Suppose coordinates on $M$ are denoted as $z$ and coordinates on $\mathfrak{g}^*$ are denoted as $\xi$.
Let $\Lambda(\mathfrak{g}^*)$ be the {bi-vector for the Lie Poisson  bracket} 
on $\mathfrak{g}^*$, in its matrix representation deriving from the same coordinates on $G$ as used to obtain $\Phi$. Then
\begin{equation}\label{FinDimPBleft}
\{F, H\} = \left(\begin{array}{cc} \nabla_z F \,{}^T & \nabla_{\xi} F\,{}^T\end{array}\right)\left(\begin{array}{cc} 0 & \Phi\\  -\Phi^T & \Lambda({\mathfrak{g}^*)}\end{array}\right)
\left(\begin{array}{c} \nabla_z H \\ \nabla_{\xi} H\end{array}\right)
\end{equation}
is a Poisson bracket on $\mathcal{C}^{\infty}(M\times \mathfrak{g}^*,\mathbb{R})$.
If the action is a right action, then \begin{equation}\label{FinDimPBright}
\{F, H\} = \left(\begin{array}{cc} \nabla_z F\,{}^T & \nabla_{\xi} F\,{}^T\end{array}\right)\left(\begin{array}{cc} 0 & -\Phi\\  \Phi^T & \Lambda({\mathfrak{g}^*)}\end{array}\right)
\left(\begin{array}{c} \nabla_z H \\ \nabla_{\xi} H\end{array}\right)
\end{equation}
is a Poisson bracket on {$M\times \mathfrak{g}^*$}.
\end{theorem}

\begin{remark} In the sequel, we denote the Poisson structure matrices in either case as $[\Lambda]$.\end{remark}

The set of examples covered by this theorem include both the canonical Darboux and the Lie Poisson brackets.
If the Lie group action is trivial, that is, $g\cdot z = z$ for all $g\in G$, then the matrix of infinitesimals is identically zero, and we have only the Lie Poisson structure on $\mathfrak{g}^*$.
Meanwhile, if the Lie group action consists of translations, then we recover the standard Darboux structure, as in the following Corollary.
\begin{corollary}
If $G=\mathbb{R}^r=M$, and the action is $(\epsilon,z)\mapsto z+\epsilon$, or in coordinates,
 \[(\epsilon_1,\dots, \epsilon_r)\cdot (z_1,\dots, z_r)=(\epsilon_1+z_1,\dots, \epsilon_r+z_r)\]
then 
\[ [\Lambda]=\left(\begin{array}{cc} 0 & I_r\\ -I_r & 0\end{array}\right)\]
where $I_r$ is the $r\times r$ identity matrix.
\end{corollary}
\begin{proof}
For this action, $\Phi$ is the identity matrix, while the Lie Poisson matrix is zero, as the Lie algebra is abelian.
\end{proof}

{
We now show that the Poisson structure given by Theorem \ref{FinDimPB} obtained from linear actions of matrix Lie groups are isomorphic to the Lie Poisson brackets of their semi-direct products.  We begin with an example.
\begin{excont} It is simpler to show the result beginning with the contragredient linear action of $G$, and to use the standard matrix representation of the
semi-direct product. 
Hence we take the action of $SL(2)$ on $M=\mathbb{R}^2$ to be,
\[ (g, (x,y)^T)\mapsto  g^{-T}(x,y)^T\]
so that the matrix of infinitesimals for the given action is
\[ \Phi = \bordermatrix{ & a&b&c\cr x& -x&0&-y\cr y & y&-x&0\cr}.\]
Then the Poisson structure given by Theorem (\ref{FinDimPB}) is
\begin{equation}\label{PBslInvTrans} \Lambda=\left(\begin{array}{ccccc} 0&0& -x&0&-y\\ 0&0& y&-x&0\\ x&y& 0& 2\xi^2&-2\xi^3\\ 0&x&-2\xi^2&0&\xi^1\\ y&0&2\xi^3&-\xi^1&0\end{array}\right).\end{equation}
We now compare this with the standard Lie Poisson matrix for the Lie algebra of $SL(2)\ltimes \mathbb{R}^2$. We use the standard matrix representation
for this Lie algebra, given by
\[ \mathbf{w}_1=\left(\begin{array}{cc} 0 & {\rm e}_1\\0&0\end{array}\right), \mathbf{w}_2=\left(\begin{array}{cc} 0 & {\rm e}_2\\0&0\end{array}\right), 
\bar{{v}}_a = \left(\begin{array}{cc} {{v}}_a & 0\\0&0\end{array}\right),\bar{{v}}_b = \left(\begin{array}{cc} {{v}}_b & 0\\0&0\end{array}\right),
\bar{{v}}_c = \left(\begin{array}{cc} {{v}}_c & 0\\0&0\end{array}\right)\]
where ${\rm e}_1=(1,0)^T$, ${\rm e}_2=(0,1)^T$, and $v_a$, $v_b$ and $v_c$ are given in Equation (\ref{sl2LAdef}).
The Lie bracket table with respect to this basis is
\[ \begin{array}{c|ccccc}
\phantom{.} [\, , \,] & w_1 & w_2 & \bar{v}_a & \bar{v}_b & \bar{v}_c\\ \hline
w_1 & 0&0& -w_1&0&-w_2\\
w_2 & 0&0& w_2& -w_1 &0\\
\bar{v}_a & w_1&-w_2&0&2\bar{v}_b & -2\bar{v}_c\\
\bar{v}_b & 0&w_1& -2 \bar{v}_b&0&\bar{v}_a\\
\bar{v}_c & w_2 & 0& 2\bar{v}_c & -\bar{v}_a&0
\end{array}\]
It can be seen that if the coefficient functions corresponding to  $\bar{v}_a$, $\bar{v}_b$ and $\bar{v}_c$ are $\xi^1$, $\xi^2$ and $\xi^3$
and those corresponding to $w_1$ and $w_2$ are $x$ and $y$ respectively, then the Poisson structure matrix for this Lie algebra will be identical to 
that in (\ref{PBslInvTrans}).  \end{excont}}

{\begin{rem} We note that in both the example and in the proof which follows, we could have started with the action on $\mathbb{R}^n$ to 
be $(g, v)\mapsto  gv$ and then shown
that the Poisson structure obtained was the same as that for Lie Poisson bracket using the contragredient representation of the semi-direct product. \end{rem}}

{
\begin{lem}\label{SemiDirProdPB} If $G\subset GL(n,\mathbb{R})$ is a Lie group, then the Poisson structure given by Theorem (\ref{FinDimPB})
for the action \[ G\times \mathbb{R}^n\rightarrow \mathbb{R}^n,\qquad (g,v)\mapsto g^{-T}v\]
is isomorphic to the Lie Poisson structure on the dual of the Lie algebra of $G\ltimes \mathbb{R}^n$.
\end{lem}
\begin{proof}
Let $(z^1,\cdots, z^n)^T$ be coordinates on $\mathbb{R}^n$
and let $\mathbf{v}_1$,\dots $\mathbf{v}_r$ be a basis of the Lie algebra $\g$ of $G$. Then, since we take the contragredient action, the matrix of infinitesimals is the $n\times r$ matrix
\[ \Phi(z) = \left( -\mathbf{v}_1^T (z^1,\cdots, z^n)^T,  \cdots,  -\mathbf{v}_r^T (z^1,\cdots, z^n)^T \right)\] and the Poisson structure matrix arising from this action is then
\begin{equation}\label{leftmatEg} \Lambda =  \left( \begin{array}{cc} 0 & \Phi(z) \\ -\Phi(z)^T & \Lambda(\g^*) \end{array}\right),\end{equation}
Let us denote the semi-direct product as $H=G\ltimes \mathbb{R}$, with $ (g,z)\ltimes (h,w) = (gh, gw+z)$. The Lie Algebra $\mathfrak{h}$ of $H$ has
a representation in $\mathfrak{gl}(n+1)$
 with basis
\[ \left\{ \mathbf{w}_j=\left(\begin{array}{cc} 0 & {\rm e}_j \\ 0&0\end{array}\right), \bar{\mathbf{v}}_i=\left(\begin{array}{cc} \mathbf{v}_i & 0 \\ 0&0\end{array}\right) \, | \, i=1,\dots r, j=1,\dots,n\right\}\]
where ${\rm e}_j=(0,\dots,0,1,0,\dots 0)^T$ where the non-zero component is in the $j$th place. We have  that $[\mathbf{w}_i, \mathbf{w}_j]=0$ while
\[ [ \mathbf{w}_j, \bar{\mathbf{v}}_{i}] = - \sum_k (\mathbf{v}_i)_{kj}\mathbf{w}_k=   \sum_k \left( -\mathbf{v}_i^T\right)_{jk} \mathbf{w}_k.\]
Further, if $ [ \mathbf{v}_i, \mathbf{v}_j]=\sum_{\ell} c^{\ell}_{ij}\mathbf{v}_{\ell}$ then $ [ \bar{\mathbf{v}}_i, \bar{\mathbf{v}}_j]=\sum_{\ell} c^{\ell}_{ij}\bar{\mathbf{v}}_{\ell}$. 
If we take the coefficient function corresponding to $\bar{\mathbf{v}}_i$ to be $\xi^i$, $i=1,\dots r$ and the coefficient function corresponding to $\mathbf{w}_j$ to be $z^j$, $j=1,\dots n$ then the Lie-Poisson bracket on $\mathfrak{h}^*$ will have for its structure matrix, that of Equation (\ref{leftmatEg}).
\end{proof}
}

Further examples of Theorem \ref{FinDimPB} will be given in \S\ref{ssec:assHam}.

\subsection{A canonical group action}
We next describe an  action of $G$ on $M\times \mathfrak{g}^*$   
which is canonical for the Poisson brackets in Theorem \ref{FinDimPB}, {that is, they define diffeomorphisms that preserve the Poisson bracket}.

We first recall our remarks concerning the Adjoint action from \S \ref{SecBasic}.
The Adjoint action of $G$ on $\mathfrak{g}$, for matrix groups and its associated matrix Lie algebra, is given by conjugations
\[ G\times \mathfrak{g}\rightarrow \mathfrak{g},\qquad (g, v)\mapsto \Ad(g)(x)=gvg^{-1},\]
and this induces the map on the associated basis $\{ \xi_1, \xi_2, \dots, \xi_r\}$ of $\g^{**}$,
given by 
\[ \boldsymbol{\xi} \mapsto \widetilde{\boldsymbol{\xi}}=\boldsymbol{\xi}\Am(g)\] 
where $\boldsymbol{\xi}=(\xi_1,\xi_2, \dots, \xi_r)$. Further, recall the Lie--Poisson matrix $ \Lambda(\mathfrak{g}^*)$ depends on  $\boldsymbol{\xi}$, specifically,
$ \Lambda(\mathfrak{g}^*)_{ij}=\sum c_{ij}^k \xi_k$. Writing $ \Lambda(\mathfrak{g}^*) = 
 \Lambda(\mathfrak{g}^*)(\boldsymbol{\xi})$ to make clear this dependence, we have
\begin{equation}\label{AdOnXiRecall} 
\Lambda(\mathfrak{g}^*) (\widetilde{\boldsymbol{\xi}})= 
\Am(g)^T\Lambda(\mathfrak{g}^*)(\boldsymbol{\xi})\Am(g)\end{equation}
by Proposition \ref{prop3things}.

\begin{theorem}\label{FinDimAdcanon}

Suppose coordinates on $M$ are denoted as $z$, so that the group action is written as $z\mapsto g\cdot z$, and coordinates on $\mathfrak{g}^*$ are denoted as $\boldsymbol{\xi}$.
Then the action of the Lie group on $M\times \mathfrak{g}^*$ given by
\[ G\times \left(M\times \mathfrak{g}^*\right) \rightarrow M\times\mathfrak{g}^*,\qquad g\cdot (z,\boldsymbol{\xi})
=\left(g^{-1}\cdot z, \boldsymbol{\xi}\Am(g)\right)\]
is canonical for the Poisson bracket (\ref{FinDimPBleft}). 
\end{theorem}
\begin{proof}
Write $g^{-1}\cdot z = \widetilde{z}$.  Then by the chain rule,
\[ \nabla_z \mapsto \nabla_{\widetilde{z}}= \left(\frac{\partial \widetilde{z}}{\partial z}\right)^{-T}\nabla_z \]
and similarly, since the action on $\boldsymbol{\xi}$ is linear, 
\[\nabla_{\boldsymbol{\xi}}\mapsto \left(\Am(g)\right)^{-1}\nabla_{\boldsymbol{\xi}}.\] 
Denote by $\mathcal{D}$ the Jacobian matrix in
\[\left(\begin{array}{c} \nabla_z\\ \nabla_{\boldsymbol{\xi}}\end{array}\right)
\mapsto  
\left(\begin{array}{cc} \left(\frac{\partial \left(g^{-1}\cdot z\right)}{\partial z}\right)^{-T} &0\\ 0
& \Am(g)^{-1}\end{array}\right)   \left(\begin{array}{c} \nabla_z\\ \nabla_{\boldsymbol{\xi}}\end{array}\right).
\]
Next, we recall Equation (\ref{PhiUltimateEqn}), 
\[
\left(\frac{\partial \left(g^{-1}\cdot z\right)}{\partial z}\right)^{-1}\Phi(g^{-1}\cdot z)= \Phi(z)\Am(g)
\]
while $\Lambda(\mathfrak{g}^*)$ transforms as in Equation (\ref{AdOnXiRecall}).

It is then straightforward to check that
 \[ \mathcal{D}^T\left(\begin{array}{cc} 0 & \Phi(g^{-1}\cdot z)\\
 -\Phi(g^{-1}\cdot z) & \Lambda(\mathfrak{g^*})(\boldsymbol{\xi}\Am(g))\end{array}\right)\mathcal{D}
 = \left(\begin{array}{cc} 0 & \Phi(z)\\
 -\Phi(z) & \Lambda(\mathfrak{g^*})(\boldsymbol{\xi})\end{array}\right)
 \]
 as required.
 \end{proof}
 
 The result is straightforward to verify in the examples. 
 \begin{excont*} \textbf{\ref{mySL2projEx} (cont.)} We have for the inverse projective action, $u\mapsto (du-b)/(-cu+a)$ with $ad-bc=1$ that $\partial_u \mapsto (cu-a)^{-2}\partial_u$.
 It can be verified directly that for the Poisson structure matrix in (\ref{newPBprojSL2}),
 \[ [\Lambda] (u, \boldsymbol{\xi})= \left(\begin{array}{c|ccc} 0& 2u&1&-u^2\\ \hline -2u&0&2\xi^2 & -2\xi^3\\-1&-2\xi^2&0&\xi^1\\u^2&2\xi^3&-\xi^1&0\end{array}\right)\]
 we have that
 \[[\Lambda] (g^{-1}\cdot u, \boldsymbol{\xi}\Am(g))=\left(\begin{array}{c|c} (cu-a)^{-2} &0\\ \hline 0& \Am(g)\end{array}\right)^T[\Lambda] (u, \boldsymbol{\xi})
 \left(\begin{array}{c|c} (cu-a)^{-2} &0\\ \hline 0&  \Am(g)\end{array}\right)\]
 where $\Am(g)$ is given in (\ref{AdgTSL2}).
 \end{excont*}
 
{\begin{rem} Lemma \ref{SemiDirProdPB} shows that linear actions of a Lie group $G$ give rise to Lie Poisson brackets 
 associated with a semidirect product, $G\ltimes \mathbb{R}^n$ for some $n$. In this case, there is a larger group of canonical group actions given by
 Adjoint action of the full semidirect product, using Proposition \ref{prop3things}, rather than just that of the $G$ component. 
 \end{rem}}
 
 \section{Compatible Poisson structures}
 {Given that a Lie group $G$ can act on a manifold $M$ using different actions, a natural question to investigate would be to study the relation between different brackets coming from different actions.}
{
 \begin{defn} We say two Poisson brackets, $\{\, , \, \}_1$ and $\{\, , \, \}_2$ are compatible if 
 \[ \{\, , \, \} := \{\, , \, \}_1+\{\, , \, \}_2\]
 is also a Poisson bracket.
 \end{defn}
 \begin{rem} An equivalent definition is:  two Poisson brackets are compatible, if their convex linear combination is also a Poisson bracket. We will use both
 notions.\end{rem}
 }

One can describe the condition on the infinitesimal matrices that determines whether or not the brackets given by Theorem \ref{FinDimPB} are compatible.
\begin{prop}   Let $G$ act on $M$ using two actions $\cdot_1$ and $\cdot_2$, with infinitesimal matrices $\Phi^1$ and $\Phi^2$. Let $\Theta^i = \pm \Phi^i$, $i=1,2$ with the sign depending on the action being left or right. Then the two Poisson brackets defined by Theorem \ref{FinDimPB} using the two actions are compatible, whenever
 \begin{equation}\label{condpoisson}
\sum_{\ell} (\Theta^1-\Theta^2)_{j,{\ell}}\frac{\partial (\Theta^1-\Theta^2)_{i,m}}{\partial z_{\ell}}  - \sum_{\ell}  (\Theta^1-\Theta^2)_{i,{\ell}}\frac{\partial (\Theta^1-\Theta^2)_{j,m}}{\partial z_{\ell}}=0,
\end{equation}
for any $i,j,m$.  \end{prop}

{\begin{proof}
The condition for a bracket defined by a matrix of the form (\ref{FinDimPBleft}) to be Poisson is well known and can be found, for example, in (6.15) at \cite{Olveryellow}. If $\Lambda(\g^\ast)= (\Lambda_{i,j})$ {is the Poisson bi-vector, and $\Theta = (\Theta^{i,j})$ is as in the statement, the bracket is Poisson if}
\begin{equation}\label{ponecond}
\left[\sum_r\frac{\partial \Lambda_{i,j}}{\partial \xi_r}\Theta^{r,m} - \sum_\ell \frac{\partial \Theta^{i,m}}{\partial z_\ell}\Theta^{j,\ell}+ \frac{\partial \Theta^{j,m}}{\partial z_\ell}\Theta^{i,\ell}\right]=0,
\end{equation}
for any $i,j,m$. To prove the proposition, we need to show that if we assume this equality to be satisfied by $\Theta^1$ and $\Theta^2$, then it will be satisfied, for example, by their average $\frac12(\Theta^1+\Theta^2$) - with the same values of $\Lambda_{i,j}$, whenever (\ref{condpoisson}) holds true. This is shown through a straightforward calculation.
\end{proof} 
In the next theorem we solve Equation (\ref{ponecond}) and classify the types of actions that give raise to compatible Poisson brackets.

\begin{theorem} {Assume we have two actions of a Lie group $G$ on a manifold $M$, and let $X^i:\g \to \mathfrak{X}(M)$ be the map associating to an element of $\g$ its infinitesimal generator as in Equation (\ref{lambdadef}), for the two actions $i=1,2$. Let $X$ be the difference vector field, namely $X = X^1-X^2:\g\to \mathfrak{X}(M)$. Then the Poisson brackets are compatible if, and only if $X(\g)$ is a commutative algebra and the actions differ by a translation.}
\end{theorem}
\begin{proof} For simplicity, let us assume that both actions are left actions (similar arguments work for the other combinations). {Let 
$\{v_j\,|\, j=1,\dots, r \}$ be the basis for $\g$ generating both infinitesimal matrices $\Phi^i$, and let $X^i_j = X^i(v_j)$ be the infinitesimal generator in the $v_j$ direction whose components define the matrices $\Phi^i$, $i=1,2$.} Formula (\ref{condpoisson}) can be described as 
\begin{equation}\label{commute}
\left[ (X_j^1- X_j^2)(\Theta^1-\Theta^2)_{i,m}  - (X^1_i- X^2_i)(\Theta^1-\Theta^2)_{j,m}\right]=0.
\end{equation}
{ Let $X = X^1-X^2$. Given that 
\[
X^i = X(v_i) = \sum_m (\Theta^1-\Theta^2)_{i,m} \frac\partial{\partial z_m},
\]
condition (\ref{commute}) is equivalent to
\[
[ X_j^1- X_j^2, X_i^1- X_i^2] = [X_i, X_j] = 0
\]
for any $i,j$. Thus, the Poisson brackets are compatible, if, and only if, the vector fields 
$\{X_i\}$ commute, as stated.}  

Assume that in a neighborhood of a point $z$ the rank of $\{X_i\}$ is constant and equal $\ell$, with dim$M=p$ and dim$G = r$. Without loss of generality, assume
$\{X_i\}_{i=1}^\ell$ are independent. Standard arguments in differential geometry tells us that there exists a coordinate system 
$\mathbf{u}$ around $z$ such that \[
X_i = \frac\partial{\partial u_i},\quad i=1,\dots, \ell, \quad\quad X_j = 0, \quad j=\ell+1,\dots,r.
\]
 From here we conclude that 
 \[
 X_i^1 = X_i^2 + \frac\partial{\partial u_i}, \quad i=1,\dots, \ell, \quad\quad X_i^1=X_i^2, \quad j=\ell+1,\dots,r.
 \]
{Therefore, locally around $z$, $g\cdot_1 u = g\cdot_2 u + w$, where $w$ is constant (it suffices to differentiate the difference between the two actions and prove that it is independent of $u$). Thus, the actions differ by {\it a translation} in the direction of $u_1,\dots, u_\ell$}.\end{proof}
 
 Notice that if the point $z\in M$ is a singularity where the rank of $\{X_i\}$ is less than $\ell$, then we can  conclude only that
 \[
 X_j = \sum_{i=1}^\ell a_j^i X_i, \quad j=\ell+1,\dots,p, \quad a_j^i = a_j^i(u_{\ell+1},\dots, u_p),
 \]
 so that
 \[
  X_j^1 = X_j^2 +\sum_{i=1}^\ell a_j^i X_i =  X_j^2 +\sum_{i=1}^r a_j^i (X_i^1-X_i^2), \quad j=\ell+1,\dots,p.
\]
If we define 
\[
\widehat{X}^s_j = X^s_j - \sum_{i=1}^\ell a_j^i X_i^s, \quad s=1,2, \quad j=\ell+1,\dots,p,
\]
then $\widehat{X}^1_j = \widehat{X}^2_j$, but these are not infinitesimal generators in general and we cannot conclude anything in terms of the group actions.

\begin{excont} A simple example of two compatible Poisson brackets in our class is given by the following
 two actions on the plane, so $M=\mathbb{R}^2$,
which differ by a translation action; 
 \[
 \begin{array}{rcl}
 g\cdot_1 x&=&\exp(\lambda)x + \mu y \\ g\cdot_1 y&=&y\end{array} \qquad 
 \begin{array}{rcl}  g\cdot_2 x&=&\exp(\lambda)x + \mu y+\epsilon\\   g\cdot_2 x&=&y+\delta\end{array}
\]
The matrix representation of the Lie group is
\[ G=\left\{ \left(\begin{array}{ccc} \exp(\lambda) & \mu &\epsilon\\0&1&\delta\\0&0&1\end{array}\right) , | \, \lambda, \mu, \epsilon, \delta \in \mathbb{R}\right\}.\]
If the Poisson structure matrices  for each action are denoted $[\Lambda_1]$ and $[\Lambda_2]$ respectively, then their convex linear sum  leads to a
one parameter family of Poisson structure matrices,
\[ \phantom{.} [\Lambda]=\left(\begin{array}{cccccc}
0&0& x& y&k&0\\
0&0&0&0&0&k\\
-x&0&0&\xi^2&\xi^3&0\\
-y&0&-\xi^2&0&0&\xi^3\\
-k&0&-\xi^3&0&0&0\\
0&-k&0&-\xi^3&0&0
\end{array}\right).
\]
It is readily checked that the Jacobi identity holds for this bracket, for all $k$.
\end{excont}

 \subsection{The resulting Hamiltonian systems}\label{ssec:assHam}
 We now consider Hamiltonian systems given by associated to the Poisson structure
 given by Theorem \ref{FinDimPB}. For a Hamiltonian function $H=H(z,\xi)$, we have the Hamiltonian system given by
 \begin{equation}\label{HamFlowEqs} \left(\begin{array}{c} \dot z \\ \dot \xi \end{array}\right) = \left(\begin{array}{c|c} 0 & \Phi \\ \hline -\Phi^T &\phantom{\Big\vert-} \Lambda({\mathfrak{g}^*}) \end{array}\right) \left(\begin{array}{c} H_z\\ H_{\xi}\end{array}\right).\end{equation}
 
  \begin{excont}  
We consider a nonlinear action of $SO(3)$ on $\mathbb{R}^2$, given in (\cite{Hydon}, Example 7.1).  If we set the coordinates on $\mathbb{R}^2$ to be $(x,y)$
 then the infinitesimal vector fields are given by
\[ \begin{array}{rcl}
X_1 &=& y\partial_x -x \partial_y\\
X_2&=&\textstyle\frac12\left(1+x^2-y^2\right)\partial_x + xy\partial_y\\
X_3&=&xy\partial_x + \textstyle\frac12\left(1-x^2+y^2\right)\partial_y.
\end{array}
\]

The Poisson structure matrix is
\begin{equation}\label{newSO3PB} \Lambda = \left(\begin{array}{ccccc}
0&0& y& \textstyle\frac12\left(1+x^2-y^2\right) & xy\\
0&0& -x & xy& \textstyle\frac12\left(1-x^2+y^2\right)\\
-y&x&0&-\xi^3 & \xi^2\\
-\textstyle\frac12\left(1+x^2-y^2\right)& -xy & \xi^3 &0&-\xi^1\\
-xy & -\textstyle\frac12\left(1-x^2+y^2\right)&-\xi^2&\xi^1&0
\end{array}\right)
\end{equation}
The Jacobi identity can be checked directly. In Figure \ref{newSO3HamPlot} we show orbits of the Hamiltonian system with
\[ H=\textstyle\frac15\left(x^2+y^2\right)+2(\xi^1)^2-(\xi^2)^2+3(\xi^3)^2\] and the Poisson structure $\Lambda$ given in (\ref{newSO3PB}), with the
unbroken curves, while the dashed line in the second plot is for $H=2\xi^1\xi^2-\left(\xi^3\right)^3$ and the Lie-Poison structure, $\Lambda(\mathfrak{so}(3)^*)$.

\begin{center}
\begin{figure}[h]
\caption{ For the Hamiltonian $H=\textstyle\frac15\left(x^2+y^2\right)+2(\xi^1)^2-(\xi^2)^2+3(\xi^3)^2$, the plots for  $(\dot z, \dot\xi)^T=\Lambda (\nabla_z H, \nabla_{\xi} H)^T$ with $\Lambda$ given in (\ref{newSO3PB}) the initial data $x(0)=y(0)=\xi_1(0)=\xi_2(0)=\xi_3(0)=1$ are shown. In (ii), the plot for the Lie Poisson system for $\mathfrak{so}(3)^*$ with
 $H=2(\xi^1)^2-(xi^2)^2+3(\xi^3)^2$, with the same initial data, is shown for comparison with the dashed line. \label{newSO3HamPlot}}
 \begin{minipage}{\textwidth}
 \begin{minipage}[b]{.5\textwidth}
\begin{center} \includegraphics[width=0.6\textwidth, trim=0 0 0 0, clip=true]{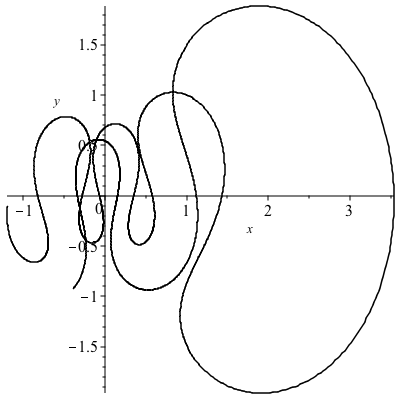} 
 
 (i) $t\mapsto (x(t),y(t))$\end{center}
 \end{minipage}
 \begin{minipage}[b]{.5\textwidth}
 \includegraphics[width=0.9\textwidth, trim=0 0 0 0, clip=true]{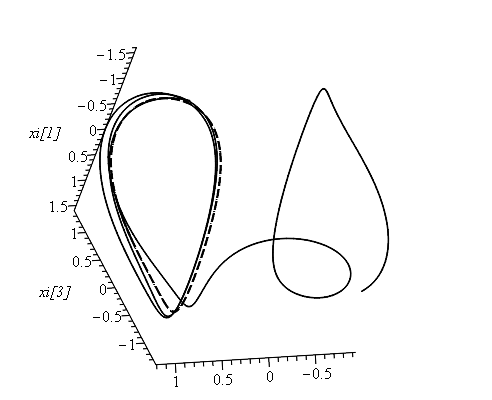}
 
 (ii) $t\mapsto (\xi^{1}(t),\xi^{2}(t),\xi^{3}(t))$ \end{minipage}
 \end{minipage}
 
 \end{figure}
 \end{center}
 \end{excont}

We can consider the Hamiltonian equations in terms of the Lie group action, at some given point $(z_0, \xi_0)$.
 Since the negative of $ \Lambda({\mathfrak{g}^*})$ is the matrix of infinitesimals for the induced Adjoint action of 
 $G$ on $\mathfrak{g}^*$ by Proposition \ref{prop3things} (2), we have  
 the following result.
 \begin{prop}\label{XiHamInvs} If $H=H(z,\xi)$ is an invariant of the Lie group action,
 \[ (z, \boldsymbol{\xi})\mapsto (g^{-1}{\cdot z}, \boldsymbol{\xi}\Am(g))\]
  that is, $H(g^{-1}\cdot z, \Ad(g)^T\xi)=H(z,\xi)$ for all $g\in G$, then for the Hamiltonian flow, Equation (\ref{HamFlowEqs}),
  \begin{equation}
  \dot{\boldsymbol{\xi}}\equiv 0.\end{equation}
 \end{prop}
 \begin{proof}
 By the invariance, we have 
 \[0=-\Phi^T \nabla_z H - \Lambda(\mathfrak{g^*})\nabla_{\boldsymbol{\xi}} H.\]
 But the right hand side is exactly $\dot{\boldsymbol{\xi}}$ for the Hamiltonian flow.
 \end{proof}
 
 \begin{excont} Consider the action of $SL(2)$ on the extended plane given as
 \[ (u,v) \mapsto \left( \frac{au+b}{cu+d}, \frac{v}{(cu+d)^2}\right)\]
 where \[ g=\left(\begin{array}{cc} a&b\\c&d\end{array}\right),\qquad ad-bc=1.\]
 Then the invariants of the action
 \[ (u,v,\boldsymbol{\xi})\mapsto \left( \frac{du-b}{-cu+a}, \frac{v}{(-cu+a)^2}, \boldsymbol{\xi}\Am(g)\right)\]
 where $\Am(g)$ is given in Equation (\ref{AdgTSL2}), are functions of
 \[ \kappa_1=4 \xi_1^2 + \xi_2\xi_3,\qquad \kappa_2=\frac1{v}\left( u^2\xi_2-u\xi_1-\xi_3\right)\]
 while the Poisson structure matrix for this Lie group action, given by Theorem \ref{FinDimPB} is
 \[ [\Lambda] = \left( \begin{array}{ccccc}
 0&0& 2u&1&-u^2\\ 0&0& 2v&0&-2uv\\ -2u&-2v&0&2\xi_2&-2\xi_3\\ -1&0&-2\xi_2&0&\xi_1\\
 u^2&2uv&2\xi_3&-\xi_1&0\end{array}\right).\]
 It is readily checked that for the Hamiltonian $H=H(\kappa_1,\kappa_2)$, Proposition \ref{XiHamInvs} holds.
 \end{excont}
 
 We can also understand the first set of equations, by comparing them to an infinitesimal action. Fix the point
 $(z_0, \boldsymbol{\xi})\in M\times \mathfrak{g}^*$. If we set $g_{\nabla_{\xi}H}(\epsilon)$ to be a smooth path in $G$ with $g(0)=e$ satisfying 
 \[ {g_{\nabla_{\xi}H}}'(0) =\sum_{i=1}^r H_{\xi_i}(z_0,\xi_0) v_i,\]
 then we have that 
  \[\dot z\big\vert_{z=z_0,\xi=\xi_0} =  \frac{{\rm d}}{{\rm d}\epsilon}\Big\vert_{\epsilon=0}  g_{\nabla_{\xi}H}(\epsilon)\cdot z_0.\]

 \subsection{Prolongation and the use of Lie group based moving frame coordinates}
 
 Lie group actions on manifolds may be prolonged to act on curves and surfaces immersed in $M$. In particular, they may 
 be prolonged to act on the jet bundle over $M$ (cf.\ \cite{Mansbook}).  If a Lie group action is locally effective on subsets, then the
 action will become free and regular after sufficient prolongation \cite{FelOl}, in which case a moving frame may be defined locally.
Many actions occurring in practice have the property that a sufficient prolongation will result in the existence of a moving frame for the action.
The frame provides co-ordinates on its domain, which look like the cartesian product of a  neighbourhood of the identity $e$ of the Lie group,
crossed with a transverse cross-section to the group orbits, which has invariants of the group action as coordinates. We will show that the 
Hamiltonian systems studied in the previous section have these invariants as constants of motion,  so that we may study the system
purely in terms of the frame variables and the coordinates of $\mathfrak{g}^*$.
 
 We recall some basic definitions and constructions; full details are given in (\cite{Mansbook} Ch 4). 
  \begin{center}
 \begin{figure}[htb]
 \caption{\label{FrameCoords} If the action is  free and regular on a domain $\mathcal{U}\subset M$, then there will be a transverse cross-section $\mathcal{K}$ to the orbits in $\mathcal{U}$, such that 
 the intersection of $\mathcal{K}$ with the orbit through a point $z$,  is a unique point, $\{k\}$. The unique element  $\sigma(z)\in G$ such that
 $\sigma(z)\cdot z =k$ defines the frame, $\sigma: M\rightarrow G$.  We have both that $\sigma(g\cdot z)=\sigma(z) g^{-1}$ for a left action, and 
 local coordinates $z=(\sigma(z), \sigma(z)\cdot z)$. }
 \includegraphics[scale=0.5,trim=20 500 100 10,clip=true]{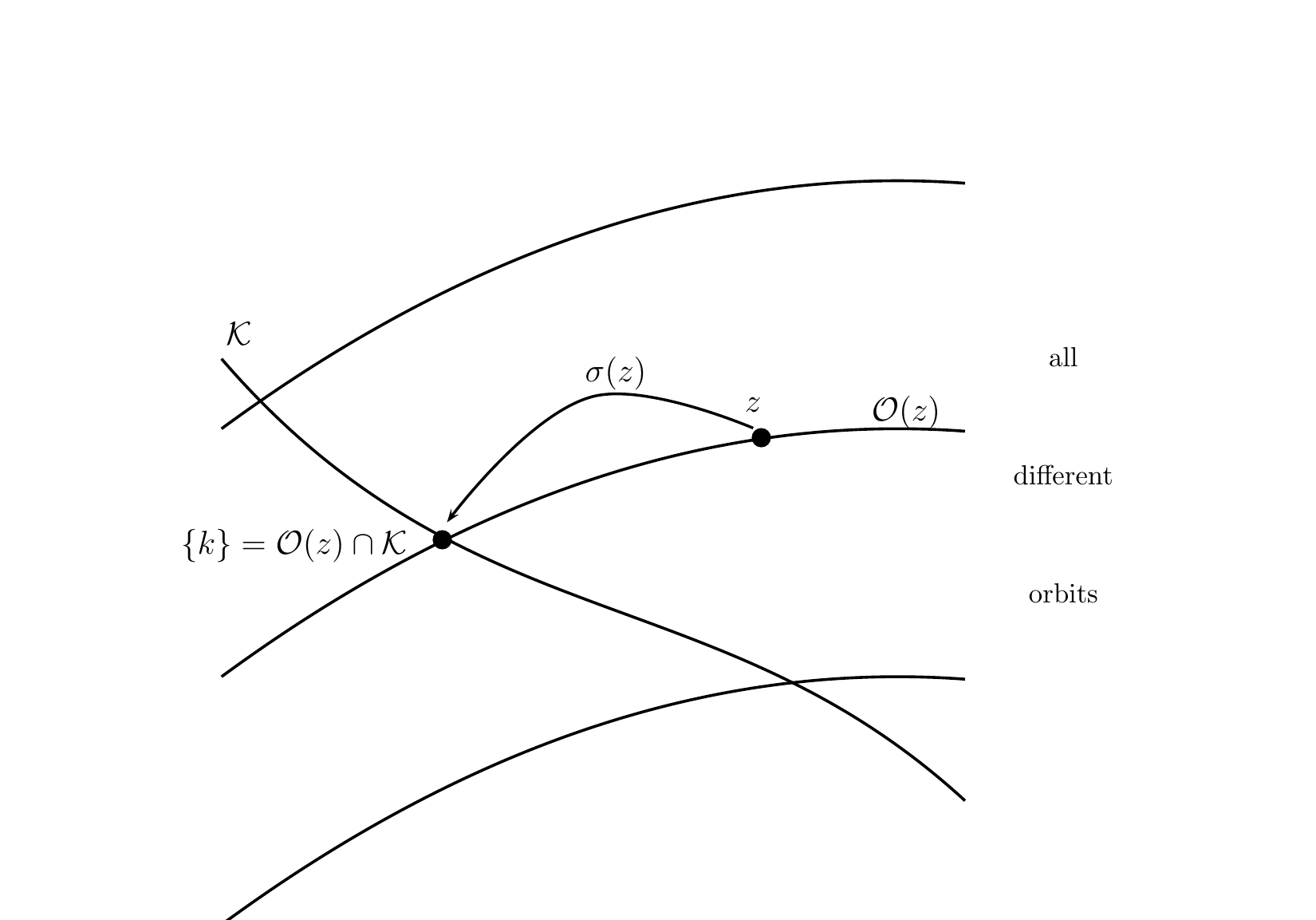}
 \end{figure}
 \end{center}
 Given a left Lie group action $G\times M\rightarrow M$ then a moving frame is an equivariant map $\sigma:M\rightarrow G$ such that
 $\sigma(g\cdot z )= \sigma(z) g^{-1}$ (a right frame), or $\sigma(g\cdot z )=g  \rho(z)$ (a left frame). A moving frame exists if the Lie group action is free and regular.
 A (local) 
 moving frame is usually calculated by setting the frame to be such that
 $\sigma(z)\cdot z \in \mathcal{K}$ where $\mathcal{K}$ is the locus of a set of equations, $\Phi(z)=0$, known as the \textit{normalisation equations}. In other words, $\sigma$ satisfies $\Phi(\sigma(z)\cdot z)=0$. The action is free and regular if the conditions for the
 Implicit Function Theorem hold for these equations. This method produces a right frame; since the Implicit Function Theorem yields a unique solution and both
 $h=\sigma(g\cdot z)$ and $h=\sigma(z)g^{-1}$ solve $\Phi(h\cdot (g\cdot z))=0$, they are equal.
  The group inverse of a right frame is a left frame.  If
 $\mathcal{K}$ is transverse to the orbits of the action, then the frame defines local coordinates. If $\mathcal{U}$ is the domain of the frame, then
 we have \[ \mathcal{U}=\mbox{dom}(\sigma) \approx G\times \mathcal{K},\qquad 
 z\mapsto (\sigma(z),\sigma(z)\cdot z),\] see Figure \ref{FrameCoords} for an illustration. 

 It can be readily seen that for a left action and a right frame, that $I(z)=\sigma(z)\cdot z$ is invariant, indeed, $I(g\cdot z) = \sigma(g\cdot z)\cdot (g\cdot z)=\rho(z) g^{-1}(g\cdot z)=\sigma(z)\cdot z=I(z)$. We denote the components of $I(z)$ as the \textit{normalised invariants}.
 
If we have a frame on (an open domain in) $M$, we can use the frame adapted coordinates {to transform the Hamiltonian equations for $z$ into equations for the evolution of the frame. This is what we do next. }

Assume for simplicity that $G$ is a matrix Lie group. To write down our results, we need some notation. Let $(a_1, \dots, a_r)$ be coordinates for $G$ in a neighbourhood of the identity $e\in G$, with
$g=g(a_1,\dots, a_r)$ being the group element with these coordinates, and we assume that
$e=g(0,0,\dots, 0)$.
If we define \[ v_i = \frac{{\rm d}}{{\rm d}\epsilon }\Big\vert_{\epsilon=0} g(0, \dots, 0,  a_i(\epsilon), 0, \dots 0)\]
then  $v_1$, \dots, $v_r$ form a basis for $\mathfrak{g}$.
Further, let the Jacobian of the map, $R_{g^{-1}}:G\rightarrow G$, $R_{g^{-1}}(h)=hg^{-1}$ at the identity element, $T_eR_{g^{-1}}$, be denoted by $\Phi_g$.

\begin{prop} \label{HamonFrame} Assume in the domain $\mathcal{U}$ that there is a frame $\sigma:\mathcal{U}\rightarrow G$ for the left action of $G$ on $M$.
Let $[\Lambda]$ be the structure matrix for the Poisson bracket defined in (\ref{FinDimPBleft}) by the group action, and let $H$ be a Hamiltonian function defined on $\mathcal{U}\times \mathfrak{g}^*$.
In the frame adapted coordinates $z\mapsto (\sigma(z), \sigma(z)\cdot z)$ on $\mathcal{U}$, we have that the Hamiltonian equations defined by $H$ and $\Lambda$ on the coordinates $ (\sigma(z), \sigma(z)\cdot z)$ are
\begin{equation}\label{FrameHamSys}\begin{array}{ccl} \displaystyle\frac{{\rm d}}{{\rm d}t}\  \sigma(z) &=& - \sigma(z)\sum \displaystyle\frac{\partial H}{\partial \xi_i} v_i\\[12pt]
\displaystyle\frac{{\rm d}}{{\rm d}t} \left(\sigma(z)\cdot z\right) &=& 0.\end{array}\end{equation}
\end{prop}
\begin{proof} 
Let $(a_1, \dots, a_r)$ be coordinates for $G$ in a neighbourhood of the identity $e\in G$. Let the frame
in these coordinates be denoted as $z\mapsto (\sigma^1(z),\sigma^2(z),\dots, \sigma^r(z))$. Assume the action of $G$ on $M$ is a left action
(the results for a right action on $M$ are similar).  The matrix of infinitesimals of the action of $G$ on $\mathcal{U}$ is, in these coordinates,
\[ \Phi = \bordermatrix{ & a_1 \dots a_r \cr
\sigma(z) & \Phi_{\sigma(z)} \cr
\sigma(z)\cdot z & 0\cr}.\]
Then the Hamiltonian system is
\[ \displaystyle\frac{{\rm d}}{{\rm d}t} \left( \begin{array}{c}\sigma(z) \\ \sigma(z)\cdot z \\ \xi\end{array}\right)
= \left( \begin{array}{c|c} \begin{array}{cc}0&0\\0&0\end{array} & \begin{array}{c} \Phi_{\sigma(z)} \\ \hline 0\end{array}\\ \hline 
\begin{array}{c|c} -\Phi_{\sigma(z)} ^T & 0 \end{array} & \Lambda(\mathfrak{g}^*)\end{array}\right) \left( \begin{array}{c}\nabla_{\sigma(z)} H \\ \nabla_{\sigma(z)\cdot z} H \\ \nabla_{\xi} H\end{array}\right)\]
where in this equation, the use of the coordinate forms of $\sigma(z)$, $\sigma(z)\cdot z$ and $\xi$ are implicit, to ease the notation.  
It is immediate that the second equation of (\ref{FrameHamSys}) holds. To see the first, we note, that
\[ \frac{{\rm d}}{{\rm d}\epsilon}\Big\vert_{\epsilon=0} \sigma(g(\epsilon)\cdot z ) = \frac{{\rm d}}{{\rm d}\epsilon}\Big\vert_{\epsilon=0}  \sigma(z) g(\epsilon)^{-1}
= -\sigma(z) g'(0)\]
so that if $g'(0)=\sum\alpha_i v_i$ then, 
\begin{equation}\label{SigEqsCompact}\frac{{\rm d}}{{\rm d}\epsilon}\Big\vert_{\epsilon=0} \sigma(g(\epsilon)\cdot z ) = -\sigma(z) \sum\alpha_i v_i.\end{equation}
But by definition of $\Phi_{\sigma(z)}$, we also have that this equation is equivalent to 
\begin{equation}\label{SigEqsOm}\frac{{\rm d}}{{\rm d}\epsilon}\Big\vert_{\epsilon=0}  (\sigma^1(g(\epsilon)\cdot z),\dots, \sigma^r(g(\epsilon)\cdot z)) =(\alpha_1, \dots, \alpha_r) \Phi_{\sigma(z)}.\end{equation}
Since the Hamiltonian equations for $\sigma(z)$ are
\begin{equation}\label{SigEqs2} \frac{{\rm d}}{{\rm d}t} \left(\begin{array}{c} \sigma^1(z)\\ \vdots\\\sigma^r(z)\end{array}\right) = \Phi_{\sigma(z)}^T \nabla_{\xi} H\end{equation}
we have, comparing (\ref{SigEqs2}), (\ref{SigEqsOm}) and (\ref{SigEqsCompact}), that the first equation in (\ref{FrameHamSys}) is simply a restatement of (\ref{SigEqs2}).
\end{proof}

We now illustrate these results with an example. 

\begin{excont*}\textbf{\ref{mySL2projEx} (continued). } We take the coordinate $u$ for $M=\mathbb{R}$ to depend on the independent variable $v$
 and we assume that $v$ is invariant under the action, so that $g\cdot v =v$. The prolongation of the action is effected by the chain rule, and is defined by
\[ g\cdot u_v = \frac{\partial (g\cdot u)}{\partial (g\cdot v)},\qquad g\cdot u_{vv} = \frac{\partial^2 (g\cdot u)}{\partial (g\cdot v)^2},\qquad 
g\cdot u_{(nv)} = \frac{\partial^n (g\cdot u)}{\partial (g\cdot v)^n}\]
which, since $g\cdot v=v$, yields
\[g\cdot u_v =  \frac{u_v}{(c u+d)^2},\qquad g\cdot u_{vv} = \frac{(c u+d) u_{vv} -2c u_v^2}{(c u+d)^3}\]
and \[ g\cdot u_{vvv} = \frac{(cu+d)^2 u_{vvv}-6c(cu+d)u_vu_{vv}+6c^2u_v^3}{(cu+d)^4}\]
to give the first three prolonged actions. It can be seen that the prolonged action on $(u,u_v, u_{vv})$-space is free and regular, indeed, we may take the normalisation equations
\[ g\cdot u=0,\qquad g\cdot u_v=1,\qquad g\cdot u_{vv}=0\]
to obtain a frame $\sigma$ on the domain $u_v>0$,
\[ \sigma:\qquad a = \frac1{\sqrt{u_v}},\qquad b=  -\frac{u}{\sqrt{u_v}},\qquad c= \frac{u_{vv}}{2 u_v ^{3/2}}\] or in the standard matrix representation for $SL(2)$,
\[ \sigma(u,u_v,u_{vv})= \left(\begin{array}{cc} \frac1{\sqrt{u_v}} &  -\frac{u}{\sqrt{u_v}} \\ \frac{u_{vv}}{2 u_v ^{3/2}} & \frac12\frac{2 u_{vv}^2 - u u_{vv}}{u_v^{3/2}}  \end{array}\right)  \]
The equivariance of the frame is demonstrated by noting that
\[ \sigma(g\cdot u,g\cdot u_v, g\cdot u_{vv})= \left(\begin{array}{cc} \frac1{\sqrt{u_v}} &  -\frac{u}{\sqrt{u_v}} \\ \frac{u_{vv}}{2 u_v ^{3/2}} & \frac12\frac{2 u_{vv}^2 - u u_{vv}}{u_v^{3/2}}  \end{array}\right)\left(\begin{array}{cc} \delta & -\beta\\ -\gamma & \alpha\end{array}\right),\qquad g = \left(\begin{array}{cc} \alpha & \beta\\ \gamma
&
\delta\end{array}\right)\]
where $\alpha\delta - \beta\gamma=1$. This equivariance then yields the matrix of infinitesimals for the action on the frame.  
If we use the frame to change coordinates from $z=(u,u_v, u_{vv},u_{vvv}, u_{4v}, \dots)$ to
\[ (\sigma(z), \sigma\cdot z)=\left(\sigma^a = \frac1{\sqrt{u_v}}, \sigma^b = -\frac{u}{\sqrt{u_v}}, \sigma^c = \frac{u_{vv}}{2 u_v ^{3/2}} , I^u_{111}=\sigma\cdot u_{vvv}, I^u_{1111}= \sigma\cdot u_{4v},
\dots \right)\]
where the $I^u_{1\cdots 1}=\sigma\cdot u_{nv}$ are the normalised invariants,
then the matrix of infinitesimals is (reverting to labelling the independent group parameters as $a$, $b$ and $c$)
\[ \Omega =\bordermatrix{ & \sigma^a &\sigma^b & \sigma^c & I^u_{111} & I^u_{1111} & \dots \cr
a &  -\sigma^a & \sigma^b & -\sigma^c & 0 & 0 & \dots \cr
b &  0& -\sigma^a & 0 &0&0&\dots\cr
c& -\sigma^b &0& -\frac{1+\sigma^b\sigma^c}{\sigma^a}&0&0&\dots}
\]
It can be seen that when this matrix is inserted into the Poisson structure matrix (\ref{FinDimPBleft}), that we obtain 
\[ \frac{{\rm d}}{{\rm d}t} I^u_{111} =0,\qquad    \frac{{\rm d}}{{\rm d}t} I^u_{1111} =0,\qquad  \frac{{\rm d}}{{\rm d}t} I^u_{1\cdots 1} =0,\]
so that these coordinates play no role, other than as constants. Thus we obtain, no matter how high we prolong the system, a six dimensional system
for the frame parameters $\sigma^a$, $\sigma^b$ and $\sigma^c$ and the coordinates of $\mathfrak{sl}(2)^*$,  $\xi_1$, $\xi_2$ and $\xi_3$.
The Poisson structure matrix for this six dimensional system is
\begin{equation}\label{SL2Sixlambda} \Lambda = \left(\begin{array}{cccccc} 0&0&0& -\sigma^a & 0 & -\sigma^b\\0&0&0& \sigma^b & -\sigma^a & 0\\ 0&0&0&-\sigma^c &0& 
-\frac{1+\sigma^b\sigma^c}{\sigma^a} \\ \sigma^a & -\sigma^b & \sigma^c &0& 2\xi_2 & -2\xi_3 \\  
0& \sigma^a & 0&-2\xi_2 &0&\xi_1 \\ \sigma^b &0& \frac{1+\sigma^b\sigma^c}{\sigma^a} & 2\xi_3 & -\xi_1&0\end{array}\right)
\end{equation}
for which the Jacobi identity may be verified directly.
Considering the resulting equations for the components of $\sigma$ for a resulting Hamiltonian system, $\dot\sigma = \Omega^T \nabla_{\xi}H$, 
noting that $\sigma^d=(1+\sigma^b\sigma^c)/\sigma^a$ and rearranging, yields
\[ \left(\begin{array}{cc} \dot\sigma^a & \dot\sigma^b\\\dot\sigma^c& \dot\sigma^d\end{array}\right)=
-\left(\begin{array}{cc} \sigma^a & \sigma^b\\ \sigma^c& \sigma^d\end{array}\right)\left(\begin{array}{cc} H_{\xi_1} & H_{\xi_2}\\ H_{\xi_3} & -H_{\xi_1}\end{array}\right)\]
verifying our remarks that a Lie group integrator may be used to integrate the Hamiltonian equations for the frame. 

Finally, we illustrate our results by considering the Hamiltonian, 
\begin{equation}\label{SL2SIxH} H=\frac15 (u^2 + u_x^2 + u_{xx}^2)+ \xi_1^2+\xi_2^2+\xi_3^2= \frac15\left( (\sigma^a)^{-4} + \frac{(\sigma^b)^2}{(\sigma^a)^2} + 4\frac{(\sigma^c)^2}{(\sigma^a)^6}\right)
+\xi_1^2+\xi_2^2+\xi_3^2,\end{equation}
with initial data $u(0)=u_v(0)=u_{vv}(0)=\xi_1(0)=\xi_2(0)=\xi_3(0)=1$ or $\sigma^a(0)=1$, $\sigma^b(0)=-1$ and $\sigma^c(0)=\textstyle\frac12$.
We plot the results in Figure \ref{SL2HamSix}. It appears (to the naked eye) that the orbit for $(\xi_1(t),\xi_2(t),\xi_3(t))$ runs from one periodic orbit to another, while the orbit for the Lie Poisson system with the same initial data is periodic, so that perhaps this orbit has split into two. In considering this system, we have made no use of the fact that $u$,v $u_v$ and $u_{vv}$ are related by differentiation with respect to $v$, and we could just as easily have called them $u$, $u_1$ and $u_2$, simply using the 
prolongation method to obtain a free and regular action and hence a frame. 
  \begin{center}
 \begin{figure}[h]
 \caption{For the Hamiltonian $H$, given in (\ref{SL2SIxH}), some plots for $(\dot \sigma, \dot\xi)^T=\Lambda (\nabla_{\sigma} H, \nabla_{\xi} H)^T$ with 
 $\Lambda$ given in (\ref{SL2Sixlambda}) and the initial data $\sigma^a(0)=1$, $\sigma^b(0)=-1$, $\sigma^c(0)=\textstyle\frac12$, $\xi_1(0)=\xi_2(0)=\xi_3(0)=1$ are shown. In Plot (i), the dashed line is for $H=  \xi_1^2+\xi_2^2+\xi_3^2$, the Lie Poisson structure $\Lambda(\mathfrak{sl}(2)^*)$, and the same initial data.
 \label{SL2HamSix}}

 \begin{minipage}{\textwidth}
 \begin{minipage}[b]{.4\textwidth}
 \includegraphics[width=1.1\textwidth, trim=0 0 0 0, clip=true]{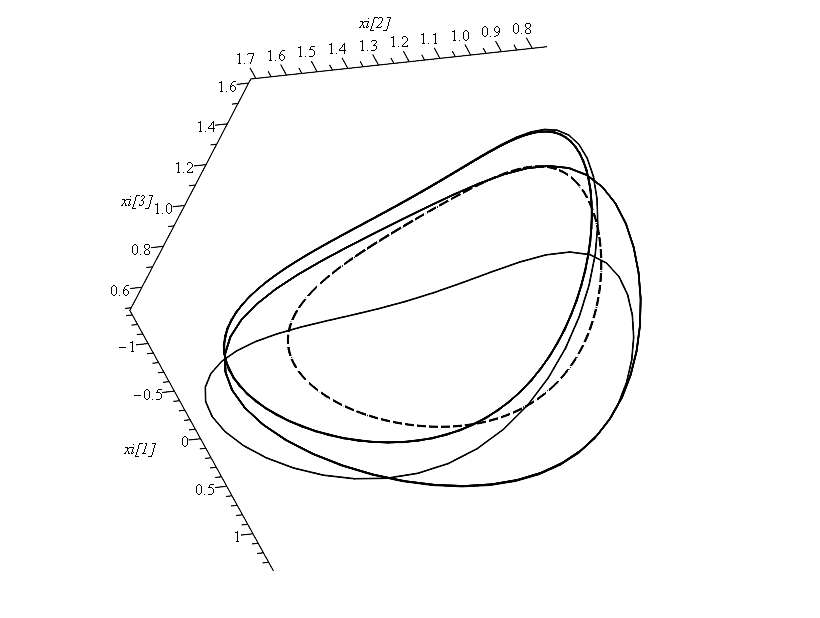} 
 
 (i) $t\mapsto t\mapsto (\xi_1(t),\xi_2(t),\xi_3(t))$ 
 \end{minipage}
 \begin{minipage}[b]{.5\textwidth}
 \includegraphics[width=\textwidth, trim=0 340 0 10, clip=true]{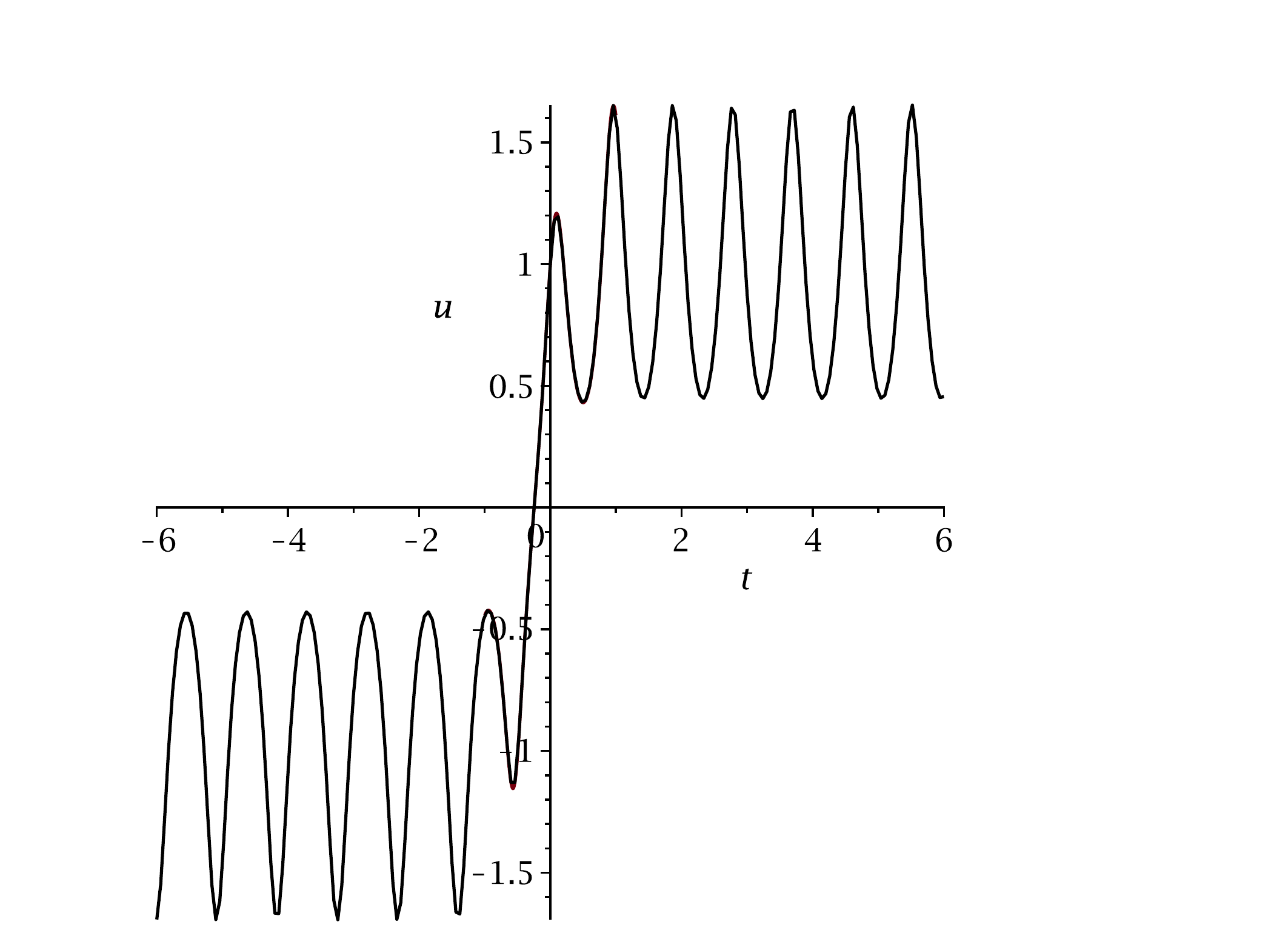}
 
 (ii) $t\mapsto u(t)= -\sigma^b(t)/\sigma^a(t)$
 \end{minipage}
 \end{minipage}
 
 \end{figure}
 \end{center}

\end{excont*}
\section{Geometric interpretation of the second bracket}\label{SecAlggeomBr}

The second Lie bracket (\ref{leftddblbracketDef}) was discussed in \cite{MKLund} in terms of a connection having zero curvature and constant torsion.
{The description which follows can be described in terms of Lie algebroid bisections (see \cite{Sch}). Instead, we present a more algebraic description, better suited to our audience.} 

The second bracket has a natural {\it geometric interpretation} that we proceed to describe next. 
We first note there is a natural product on the group of sections $A^1(M,M\times G)$ on $M$, defined by
\begin{equation}\label{NatProdA1MG}
(z,g(z))\cdot_{\mbox{nat}} (z,h(z)) = (z, g(z)h(z)).
\end{equation}

Recall the definition of an action, Definition \ref{ActionDef}. Assume we have an action of $G$ on $M$ given by
\begin{eqnarray*}
G\times M \to& M \\ (g,s) \to& \lambda(g,s) = g\cdot s
\end{eqnarray*}
One might think that this action induces a natural action of the group of sections $A^1(M,M\times G)$ on $M$, defined by $(g,s) \to \lambda(g(s),s) = g(s)\cdot s$. But this is not an action with respect to the natural product, given in
Equation (\ref{NatProdA1MG}), since, if $\lambda$ is, for example, a left action, one has
\[
(gh,s)\to \lambda(g(s)h(s), s) = \lambda(g(s),\lambda(h(s),s))
\]
while
\[
(g,(h,s)) \to (g,\lambda(h(s),s))\to \lambda(g(\lambda(h(s),s)), \lambda(h(s),s)).
\]
On the other hand, we can define a different product that will give $A^1(M,M\times G)$ a structure of local Lie group (close to the identity), and will allow us to define a natural local action on $M$. If $\lambda$ is a left action of $G$ on $M$, define
\begin{eqnarray*}
A^1(M, G)\times A^1(M, G) \to& A^1(M, G), \\ (g(s), h(s)) \to& (g\ast h)(s) = g(\lambda(h(s), s)) h(s),
\end{eqnarray*}
while if $\lambda$ is a right action, define 
\begin{eqnarray*}
A^1(M, G)\times A^1(M, G) \to& A^1(M, G), \\ (g(s), h(s)) \to& (g\ast h)(s) = g(s) h(\lambda(g(s), s)).
\end{eqnarray*}
\begin{prop} \label{prop1} The product $\ast$ is associative. If $e(s)$ is the unit section with $e(s)=e$ for all $s\in M$, then $(e\ast h)(s) = (h\ast e)(s) = h(s)$, and in a neighborhood of $e\in A^1(M, G)$, with respect to the $C^1$-topology in $A^1(M, G)$, the inverse element exists {although the map $g \to g^{-1}$ is not smooth in general}. The map
\begin{eqnarray*}
A^1(M, G)\times M \to& M\\ (g,s) \to& \sigma(g,s) = \lambda(g(s), s)
\end{eqnarray*}
is an {\it action} {of the formal Lie group $A^1(M, G)$} on $M$, with the same parity as the original action of $G$ on $M$.
\end{prop}
\begin{proof}
The property $e\ast h = h \ast e = h$ is trivially satisfied. Associativity also follows from
\[
((g\ast h)\ast f)(s) 
\]
\[= (g\ast h)(\lambda(f(s),s)) f(s) = g(\lambda(h(\lambda(f(s),s)),\lambda(f(s),s)) h(\lambda(f(s),s)) f(s) 
\]
\[
= g(\lambda(h(\lambda(f(s),s))f(s), s))h(\lambda(f(s),s)) f(s) = g(\lambda((h\ast f)(s), s)) (h\ast f)(s) 
\]
\[
= g\ast(h\ast f)(s).
\]
if a left action, or from
\[
((g\ast h)\ast f)(s) 
\]
\[= (g\ast h)(s) f(\lambda((g\ast h)(s), s)) 
= g(s) h(\lambda(g(s), s)) f(\lambda(g(s)h(\lambda(g(s), s))), s)) 
\]
\[
= g(s) h(\lambda(g(s), s)) f(\lambda((h(\lambda(g(s), s))), \lambda(g(s), s))) = g(s) (h\ast f)(\lambda(g(s), s)) 
\]
\[= (g\ast(h\ast f))(s)
\] 
if a right action. 

Since $e(s)\cdot s = s$, {the finite dimensional} inverse function theorem guarantees that $g(s)\cdot s$ is invertible for $g$ in a $C^1$-neighborhood of $e$, {meaning whenever 
\[
\mathrm{sup}_{s\in M}\{|g(s)-e(s)|, |Dg(s)-De(s)|\} .
\]
is small, where $|\cdot|$ is any metric in $G$ (we can assume that $G\subset \mathrm{GL}(n)$)}. {Therefore, the relation $g^{\ast-1}(g(s)\cdot s) = g^{-1}(s)$ allows us} to define a $\ast$-inverse of $g$ around $e(s)$.

Finally, if the action of $G$ is left
\[
\sigma(g\ast h,s) = \lambda((g\ast h)(s), s) = \lambda(g(\lambda(h(s),s)))h(s), s) 
\]
\[
= \lambda(g(\lambda(h(s),s)),\lambda(h(s),s)) = \sigma(g,\sigma(h,s)). 
\]
If the action is right
\[
\sigma(g\ast h,s) = \lambda((g\ast h)(s), s) = \lambda(g(s)h(\lambda(g(s),s)), s) 
\]
\[
= \lambda(h(\lambda(g(s),s)),\lambda(g(s),s)) = \sigma(h,\sigma(g,s)) 
\]
and so we have an action of $A^1(M,M\times G)$ on $M$ with the same parity as the original.
\end{proof}}

{Consider now the $\ast$-conjugation
\begin{eqnarray*}
A^1(M, G)\times A^1(M, G) \to& A^1(M, G), \\ (g(s), h(s)) \to& (g\ast h\ast g^{\ast-1})(s) 
\end{eqnarray*}
where $g^{\ast-1}$ is the (local) inverse of $g$ with respect to the $\ast$ product.}

{\begin{theorem} Assume $G$ acts  on $M$ and denote by $A^1(M,M\times G)$ be the local {formal} Lie group with operation $\ast$ {as described in proposition \ref{prop1}}. Then, the Lie bracket of its Lie algebra $A^1(M, M\times \g)$ is defined by $-\llbracket \,  \rrbracket$.
\end{theorem}}
{Notice that  $G$ is a formal Lie group, and hence we do not have a guarantee that the operations are smooth. The reason for it is that even though $g^{\ast-1}$ exists, to guarantee the smoothness of $g \to g^{\ast-1}$ we will need to apply an infinite dimensional inverse function theorem, theorems that, when they exist, require strong conditions and thus depend very much on each individual case. For example, different arguments may be applied when either $M$ or $G$ is compact, (A.\ Schmeding, (2018) Private communication). Nevertheless, in order to differentiate to define the associated Lie algebra, we only need to guarantee that the map $g_\epsilon \to g^{\ast-1}_\epsilon$ is smooth, for a one-parameter family $g_\epsilon$ such that $g_0 = e$ and $\frac d{d\epsilon}|_{\epsilon=0}g_\epsilon = v\in A^1(M, \g)$. This is guaranteed by the standard finite dimensional inverse function theorem under the hypothesis of proposition \ref{prop1}.}
{\begin{proof}
Assume $G$ acts on the left on $M$, and let $h(\epsilon, s)$ be a curve in $A^1(M,M\times G)$ with $h(0,s) = e$ and $\frac d{d\epsilon}|_{\epsilon=0}h(\epsilon,s) = x(s)\in A^1(M, M\times \g)$. If we differentiate $g\ast h\ast g^{\ast-1}(s) = g\left([h(g^{\ast-1}(s)\cdot s)]\cdot g^{\ast-1}(s)\cdot s\right)h(g^{\ast-1}(s)\cdot s) g^{\ast-1}(s)$ with respect to $\epsilon$ at $\epsilon=0$, we obtain the $\ast$-Adjoint action of $A^1(M, G)$ on $A^1(M, \g)$, given by
\begin{eqnarray*}
A^1(M, G)\times A^1(M, \g) \to& A^1(M, \g), \\ (g(s), x(s)) \to& g(g^{\ast-1}(s)\cdot s) x(g^{\ast-1}(s)\cdot s) g^{\ast-1}(s)\\&+ T_R{g^{\ast-1}(s)} \left(\mathcal{L}_{X(x(g^{\ast-1}(s)\cdot s))}g\right))(g^{\ast-1}(s)\cdot s),
\end{eqnarray*}
where $T_R{g(s)}$ denotes the derivative of the right multiplication by $g(s)$. Notice that $g(s)\ast g^{\ast-1}(s) = g(g^{\ast-1}(s)\cdot s)g^{\ast-1}(s) = e$, and so $g(g^{\ast-1}(s)\cdot s)= (g^{\ast-1})^{-1}(s)$. Therefore, the $\ast$-Adjoint is given by
\[
\mathcal{A}d_\ast(g)(x)(s) = \Ad((g^{\ast-1})^{-1}(s))(x(g^{\ast-1}(s)\cdot s)) + T_R{g^{\ast-1}(s)} \left(\mathcal{L}_{X(x(g^{\ast-1}(s)\cdot s))}g\right))(g^{\ast-1}(s)\cdot s).
\]
Next, assume we differentiate in the direction of $g(s)$. Let $g(\epsilon, s)$ be a curve in $A^1(M,M\times G)$ such that $g(0,s)=e$ for all $s$ and $\frac{dg}{d\epsilon}|_{\epsilon=0}(s)
= y(s)\in A^1(M, M\times \g)$. Using that $g^{\ast-1}\ast g = e$ and so $g^{\ast-1}(\epsilon, g(\epsilon, s)\cdot s) = g^{-1}(\epsilon,s)$, differentiating both sides we get
\[
\frac{\partial g^{\ast-1}}{\partial\epsilon}|_{\epsilon=0}(s) + \mathcal{L}_{X(y)}(g^{\ast-1}(0,s)) = \frac{\partial g^{\ast-1}}{\partial\epsilon}|_{\epsilon=0}(s) = -y(s)
\]
since $g(0,s) = e$ implies that $g^{\ast-1}(0,s) = e$ also.
We have 
\[
\frac d{d\epsilon}|_{\epsilon=0} \Ad((g^{\ast-1})^{-1}(s))(x(g^{\ast-1}(s)\cdot s))  = [y(s), x(s)] - (\mathcal{L}_{X(y)} x)(s),
\]
while, since  $g(0,s)$ is constant, all the terms in
\[
\frac d{d\epsilon}|_{\epsilon=0}T_R{g^{\ast-1}(s)} \left(\mathcal{L}_{X(x(g^{\ast-1}(s)\cdot s))}g\right))(g^{\ast-1}(s)\cdot s)
\]
will vanish due to the Lie derivative of $g$ being zero as $\epsilon = 0$, except for the one involving differentiation of $g$, namely
\[
\left(\mathcal{L}_{X(x(s))}\frac{\partial g}{\partial\epsilon}|_{\epsilon=0}\right)(s) = \mathcal{L}_{X(x(s))} y.
\]
Therefore, the adjoint action of $A^1(M, \g)$ on itself (the Lie bracket), is given by 
\begin{eqnarray*}
A^1(M, \g)\times A^1(M, \g) \to& A^1(M, \g), \\ (y(s), x(s)) \to& [y,x]-\mathcal{L}_{X(y)} x + \mathcal{L}_{X(x)} y = -\llbracket y , x \rrbracket
\end{eqnarray*}
A similar process proves the case of a right action of $G$ on $M$.\end{proof}}


{\begin{rem} It is worth remarking that the structure of $\ast$-Lie group is not global in general and $\ast$-inverses might not exist even if $g(s)$ is $C^0$-close to $e(s)$. Indeed, it suffices to consider a left action of $G$ on $M$, and a right moving frame for the action associated to a cross section $\mathcal{C}$. That is $g(s)$, $s\in M$ satisfies $g(s)\cdot s\in \mathcal{C}$ and  $g(f\cdot s) = g(s) f^{-1}$, for all $s$ and all $f\in G$.  
In general moving frames might only exist locally, but on the neighborhood where they exist, one would have that $(g\ast h)(s) = g(h(s)\cdot s) h(s) = g(s) h^{-1}(s) h(s) = g(s)$, for {\it any} $h(s) \in A^1(M, G)$. That is, $g$ has no $\ast$-inverse.
\end{rem}}

\section{Infinite dimensional Poisson brackets}
In this chapter we will assume that $M=S^1$ or $\R$. The study of the possible applications of the second bracket to the study of completely integrable PDEs is under way in this case. Still, preliminary results point at a possible use and we are including them here, in our last section. In particular, we are able to prove that the standard cocycle used to construct a central extension of the algebra $C^\infty(S^1, \g)$ is also a cocycle when we consider $\g$ endowed with the second bracket, rather than the first. Furthermore, standard arguments allow us to define a companion bracket to the Lie-Poisson bracket associated to $\llbracket\, , \, \rrbracket$. These companions are often associated to bi-Hamiltonian structures for well known completely integrable PDEs (e.g. KdV in the case of $\g= \mathfrak{sl}(2,\R)$). We are currently studying under which conditions this Hamiltonian pencil can be used to integrate PDE's.
\subsection{Central extensions} a central extension of a Lie algebra $\g$ is given by $\g\oplus \R$, with the standard algebra structure and the Lie bracket
\begin{equation}\label{CentExtBr}
[(x,l), (y,t)] = ([x,y], \beta(x,y)).
\end{equation}
where the map 
 \[\beta: \mathfrak{g}\times \mathfrak{g} \rightarrow \mathbb{R}\]
 is skew-symmetric, bilinear, and  satisfies what is known as a cocycle condition,
\[ \beta( w_1,[w_2,w_3])+\beta(w_2,[w_3,w_1])+\beta(w_3,[w_1,w_2])=0.\]
Under these conditions, (\ref{CentExtBr}) does indeed define a Lie bracket, being bilinear, skew-symmetric and satisfying the Jacobi identity. Any linear map  $\theta:\mathfrak{g}\rightarrow \mathbb{R}$ defines
a central extension, with $\beta(x,y)=\theta([x,y])$. Such extensions are called trivial, or coboundaries.

The dual to this Lie algebra can be identified with $\g^\ast\oplus\R$, and the action on the central extension is defined as
\[
(\xi, l)(x,t) = \xi(x)+lt.
\]
\subsubsection{Lie-Poisson bracket for a central extension}
Assume $F$ and $H$ are two functions on $\g^\ast\oplus \R$. The variational derivative of $F$ and $H$ are given using the relation
\[
\frac d{d\epsilon}|_{\epsilon=0} F((\xi,t)+\epsilon (\nu,r)) = (\nu,r)(\delta_\xi F, \delta_tF) = \nu(\delta_\xi F)+r\delta_tF
\]
for all $\nu\in \g^\ast$, and where $\xi\in \g^\ast$ and $\delta_\xi F\in \g$, $\delta_tF\in \R$. The Lie-Poisson bracket is thus defined as
\begin{equation}\label{Poisson}
\{F, H\}(\xi,r) = (\xi,r)\left( [(\delta_\xi F, \delta_tF), (\delta_\xi H, \delta_tH)]\right) = \xi([\delta_\xi F, \delta_\xi H])+r\beta(\delta_\xi F,\delta_\xi H).
\end{equation}
Notice that once we fix $r$, we have a Poisson bracket on $\g^\ast$. It is known that the $r$-family of Poisson brackets are all equivalent, except for $r=0$. It is customary to assume that $r=-1$, or $r=1$, depending on the case, but we can choose any constant. 

\subsubsection{First bracket}
Assume that the bracket in our algebra is the standard commutator, $[,]$, and assume that $\g$  is semisimple so that we can identify $\g$ and $\g^\ast$ using the trace inner product, assuming that they have a representation as matrices. That is, if $\xi \in \g^\ast$, then there exists $x^\xi\in \g$, such that
\[
\xi(x) = {\rm tr}(x^\xi x)
\]
for all $x\in \g$.  If we consider, for example, $M=S^1$ and smooth sections $C^\infty(S^1,\g)$ and $C^\infty(S^1,\g)^\ast$ instead finite dimensional algebras, the inner product would be
\[
\xi(x) = \int_M{\rm tr}(x^\xi x).
\]
Instead of $S^1$ we could choose smooth functions on $\R$ vanishing at infinity, or other conditions that will ensure that boundary conditions during integration will vanish. Let us denote by $s$ the coordinates in $M$ ($S^1$ or $\R$). It is known that $\beta$ defined as 
\[
\beta(x,y) = \int_M {\rm tr}(x y_s)
\]
is a cocycle for our first bracket.
 In this case, the Poisson structure (\ref{Poisson}) becomes
\[
\{F, H\}_1(\xi,r) = \xi([\delta_\xi F, \delta_\xi H])+r\beta(\delta_\xi F,\delta_\xi H) 
\]
\[
= \int_M{\rm tr}\left(x^\xi\left[\delta_\xi F, \delta_\xi H\right] + r\delta_\xi F(\delta_\xi H)_s\right) ds
=  \int_M{\rm tr}\left(\left(\left[x^\xi,\delta_\xi F\right] - r(\delta_\xi F)_s\right)\delta_\xi H\right) ds.
\]
If we choose $r=-1$, then we have
\begin{equation}\label{firstbr}
\{F, H\}_1(\xi,r) = \int_M{\rm tr}\left(\left((\delta_\xi F)_s+ \left[x^\xi,\delta_\xi F\right]\right)\delta_\xi H\right) ds.
\end{equation}
Since this is true for all $H$, the Hamiltonian vector field for $F$ can be identified with $(\delta_\xi F)_s+ \left[x^\xi,\delta_\xi F\right]$ and the Hamiltonian evolutions with $(x^\xi)_t = (\delta_\xi F)_s+ \left[x^\xi,\delta_\xi F\right]$. 

\subsubsection{Second bracket}

\begin{prop} The map $\beta$ is also a cocycle when $C^\infty(S^1, \g)$ is endowed with the second bracket 
$\llbracket\, , \, \rrbracket$.
\end{prop}
\begin{proof}
It suffices to show that 
\[
\beta(\llbracket y,z\rrbracket, x)+\beta(\llbracket z,x\rrbracket,y)+\beta(\llbracket x,y\rrbracket,z) = 0
\]
and, since $\beta$ is a cocycle for the first bracket, that
\[
\int_M\mathrm{tr}(\rho(y)(z) x_s- \rho(z)(y) x_s) + \mathrm{tr}(\rho(z)(x) y_s- \rho(x)(z) y_s)+\mathrm{tr}(\rho(x)(y) z_s- \rho(y)(x) z_s) = 0. 
\]
Now, if dim$M=1$, we have that $\rho(x)(y) = \hat\rho(x,s) y_s$, for some {\it scalar} $\hat\rho(x,s)$. This implies that the above integrand is equal to
\[
\mathrm{tr}(\hat\rho(y,s)z_s x_s- \hat\rho(z,s)y_s x_s + \hat\rho(z,s)x_s y_s- \hat\rho(x,s)z_s y_s+\hat\rho(x,s)y_s z_s- \hat\rho(y,s)x_s z_s)
\]
\[
=\mathrm{tr}\left(\hat\rho(y,s)[z_s, x_s]+ \hat\rho(z,s)[x_s, y_s]+\hat\rho(x,s)[y_s, z_s]\right)
\]
\[
=\hat\rho(y,s)\mathrm{tr}([z_s, x_s])+ \hat\rho(z,s)\mathrm{tr}([x_s, y_s])+\hat\rho(x,s)\mathrm{tr}([y_s, z_s]) = 0
\]
since the trace of the commutator vanishes.
\end{proof}
Notice that the scalar $\hat\rho(x,s)$ is linear in $x$. Therefore, since $\g$ is semisimple, we can write it as 
\begin{equation}\label{E}
\hat\rho(x,s) = \mathrm{tr}(xE(s))
\end{equation}
for some matrix $E(s)\in \g$ depending on $s\in M$. Notice also that the inner product defined by the trace is not invariant under the Adjoint action associated to $\llbracket\, , \, \rrbracket$. This fact will cause the appearance of extra terms in a typical Hamiltonian evolution.

Using that $\beta$ is also a cocycle for $\llbracket\, , \, \rrbracket$, we  can write the Lie-Poisson structure associated to the central extension of $\llbracket\, , \, \rrbracket$ as
\[
\{F, H\}(\xi,r) = \int_M{\rm tr}\left(x^\xi\left[\left[\delta_\xi F, \delta_\xi H\right]\right] + r\delta_\xi F(\delta_\xi H)_s\right) ds
\]
\[
=\int_M{\rm tr}\left(x^\xi\left(\left[\delta_\xi F, \delta_\xi H\right] + {\rm tr}(\delta_\xi F E(s))(\delta_\xi H)_s - {\rm tr}(\delta_\xi H E(s))(\delta_\xi F)_s\right) + r\delta_\xi F(\delta_\xi H)_s\right) ds.
\]

\[
 =\int_M{\rm tr}\left([x^\xi,\delta_\xi F] \delta_\xi H+ {\rm tr}((\delta_\xi F) E)x^\xi(\delta_\xi H)_s - {\rm tr}((\delta_\xi H) E)x^\xi(\delta_\xi F)_s + r\delta_\xi F(\delta_\xi H)_s\right) ds
\]
\[
=\alpha\int_0^{2\pi}{\rm tr}\left(\left([x^\xi,\delta_\xi F] - r (\delta_\xi F)_s -\left( {\rm tr}((\delta_\xi F) E)x^\xi\right)_s\right)\delta_\xi H - {\rm tr}((\delta_\xi H) E)x^\xi(\delta_\xi F)_s \right) ds
\]
Finally, we have
\[
{\rm tr}\left({\rm tr}((\delta_\xi H) E)x^\xi(\delta_\xi F)_s \right) = {\rm tr}(E(\delta_\xi H)){\rm tr}\left(x^\xi(\delta_\xi F)_s \right)= {\rm tr}({\rm tr}\left(x^\xi(\delta_\xi F)_s \right)E(\delta_\xi H)).
\]
From here, the bracket can be written as
\begin{equation}\label{secondbr}
\{F, H\}(\xi,r) = \alpha\int_0^{2\pi}{\rm tr}\left(\left\llbracket x^\xi,\delta_\xi F] - r (\delta_\xi F)_s -\left( {\rm tr}((\delta_\xi F) E)x^\xi\right)_s - {\rm tr}\left(x^\xi(\delta_\xi F)_s \right)E \right]\delta_\xi H \right) ds
\end{equation}
and so, if $r=-1$, the Hamiltonian vector field can be identified with the element of the algebra on the left and the Hamiltonian evolution would be
\[
(x^{\xi})_t =  (\delta_\xi F)_s +[x^\xi,\delta_\xi F]  -\left( {\rm tr}((\delta_\xi F) E)x^\xi\right)_s - {\rm tr}\left(x^\xi(\delta_\xi F)_s \right)E.
\]

The following result is known for general Lie algebras, but we reproduce the proof to make the paper self-contained. 
{\begin{theorem}
Consider the bracket
\begin{equation}\label{zerobr}
\{F,H\}_0(\xi) = \alpha\int_0^{2\pi}{\rm tr}\left( x^{\xi_0} \llbracket\delta_\xi F, \delta_\xi H\rrbracket\right) ds,
\end{equation}
where $\xi_0\in \g^\ast$ is constant. The bracket (\ref{zerobr}) is Poisson and compatible with (\ref{secondbr}) for any value of the central parameter $r$ and any choice of $\xi_0$.
\end{theorem}}
{\begin{proof}
First of all, recall  that the second variation of a functional $F(\xi)$ defined on $\g^\ast$ is a symmetric bilinear form on $\g^\ast$, which we will denote by $D^2F(\xi)(\ast,\ast)$. In that sense, $D^2F(\xi)(\eta,\ast)$ is a linear map on $\g^\ast$, and can be identified with an element of $(\g^\ast)^\ast = \g$, as follows
\[
\nu(D^2 F(\xi)(\eta,\ast)) = D^2 F(\xi)(\eta,\nu) = \eta(D^2 F(\xi)(\nu,\ast))
\]
for any $\xi,\eta,\nu\in\g^\ast$. To make calculations easier to read, let us denote
\[
\nu(x) = \langle \nu,x\rangle
\]
for any $\nu\in\g^\ast, ~x\in \g$. In that case, the bracket can be written as 
\[
\{F,H\}_0(\xi) = \alpha\int_0^{2\pi}\langle \xi_0, \llbracket\delta_\xi F, \delta_\xi H\rrbracket\rangle ds.
\]
Since $\llbracket \, , \, \rrbracket$ is a Lie bracket, we have that 
\[
D\llbracket\delta_\xi F, \delta_\xi H\rrbracket(\xi)(\eta) = \llbracket D^2F(\xi)(\eta, \ast), \delta_\xi H\rrbracket+ \llbracket\delta_\xi F, D^2H(\xi)(\eta,\ast)\rrbracket.
\]
Therefore
\[
\int_0^{2\pi}\langle\eta,\delta_\xi\{F,H\}_0\rangle ds =\int_0^{2\pi}\langle \xi_0, \llbracket D^2F(\xi)(\eta, \ast), \delta_\xi H\rrbracket+ \llbracket\delta_\xi F, D^2H(\xi)(\eta,\ast)\rrbracket\rangle ds\]
\[
=\int_0^{2\pi}\langle \llbracket\delta_\xi H,\xi_0\rrbracket^\ast, D^2F(\eta,\ast)\rangle - \langle\llbracket\delta_\xi F,\xi_0\rrbracket^\ast, D^2H(\xi)(\eta,\ast)\rangle ds
\]
\[
=\int_0^{2\pi} \left( D^2F(\eta,\llbracket\delta_\xi H,\xi_0\rrbracket^\ast) - D^2H(\xi)(\eta,\llbracket\delta_\xi F,\xi_0\rrbracket^\ast)\right) ds
\]
\[
=\int_0^{2\pi}\langle\eta,D^2F(\llbracket\delta_\xi H,\xi_0\rrbracket^\ast,\ast) - D^2H(\xi)(\llbracket\delta_\xi F,\xi_0\rrbracket^\ast,\ast)\rangle ds
\]
where $\llbracket\, , \, \rrbracket^\ast$ is the dual of $\llbracket\, , \, \rrbracket$. From here we conclude that 
\[
\delta_\xi\{F,H\}_0 = D^2F(\llbracket\delta_\xi H,\xi_0\rrbracket^\ast,\ast) - D^2H(\xi)(\llbracket\delta_\xi F,\xi_0\rrbracket^\ast,\ast).
\]
We can now calculate Jacobi's identity. If $F,G,H$ are three functionals on $A^1(M,\g^\ast)$, we have
\[
\{\{F,H\}_0,G\}_0(\xi) = \alpha\int_0^{2\pi}\langle \xi_0, \llbracket\delta_\xi\{F,H\}_0, \delta_\xi G\rrbracket\rangle ds = \alpha\int_0^{2\pi}\langle\llbracket\delta_\xi G,\xi_0\rrbracket^\ast,\delta_\xi\{F,H\}_0\rangle ds
\]
\[
=\alpha\int_0^{2\pi}\langle\llbracket\delta_\xi G,\xi_0\rrbracket^\ast,D^2F(\llbracket\delta_\xi H,\xi_0\rrbracket^\ast,\ast) - D^2H(\xi)(\llbracket\delta_\xi F,\xi_0\rrbracket^\ast,\ast)\rangle ds 
\]
\[= \alpha\int_0^{2\pi}D^2F(\llbracket\delta_\xi H,\xi_0\rrbracket^\ast,\llbracket\delta_\xi G,\xi_0\rrbracket^\ast) - D^2H(\xi)(\llbracket\delta_\xi F,\xi_0\rrbracket^\ast,\llbracket\delta_\xi G,\xi_0\rrbracket^\ast) ds.
\]
Using the symmetry of the second variation, we can conclude that the circular sum of this expression vanishes. A similar calculation shows that the sum of (\ref{secondbr}) and (\ref{zerobr}) is also Poisson (notice that adding both brackets merely translates $x^\xi$ to $x^{\xi} + x^{\xi_0}$ in the expression of (\ref{secondbr}), so the proof that the sum is Poisson is almost identical to that of (\ref{secondbr}) being Poisson).
\end{proof}}

\end{document}